\documentclass{article}

\usepackage{arxiv}
\usepackage[utf8]{inputenc} 
\usepackage[T1]{fontenc}    
\usepackage{url}            
\usepackage{booktabs}       
\usepackage{amsfonts}       
\usepackage{nicefrac}       
\usepackage{microtype}      
\usepackage{graphicx}
\usepackage{stmaryrd} 
\usepackage{xspace}
\usepackage{arydshln}  
\usepackage{doi}
\usepackage{bm}
\usepackage{mathtools}
\usepackage{amsmath}
\usepackage{amsthm}
\usepackage{hyperref}
\usepackage{cleveref}
\usepackage{algorithm,algpseudocode}
\newcommand{\email}[1]{\href{mailto:#1}{#1}}

\newtheorem{assumption}{Assumption}
\newtheorem{lemma}{Lemma}
\newtheorem{theorem}{Theorem}
\newtheorem{corollary}{Corollary}
\crefname{algorithm}{algorithm}{algorithms}
\Crefname{algorithm}{Algorithm}{Algorithms}

\newcommand{\LameName}{Lam\'e\xspace}

\newcommand{\norm}[1]{\lVert#1\rVert}
\newcommand{\modulus}[1]{|#1|}
\newcommand{\jump}[1]{\llbracket #1 \rrbracket}

\newcommand{\ujump}[0]{\llbracket \disp \rrbracket}
\newcommand{\ujumpn}[0]{\llbracket {u}_n \rrbracket}
\newcommand{\ujumpt}{\llbracket \dot{\disp}_\tau \rrbracket}
\newcommand{\Lame}[0]{\lambda_\text{\textnormal{\LameName}}}

\newcommand{\Cfrac}{{\Omega_l}}
\newcommand{\perm}{\mathbf{K}}
\newcommand{\visc}{\mu_{f}}
\newcommand{\lambdat}{{\bm{\lambda}_\tau}}
\newcommand{\lambdan}{{{\lambda}_n}}
\newcommand{\lambdafull}[0]{\bm{\lambda}}

\newcommand{\disp}{\mathbf{u}}

\renewcommand\vec{\mathbf}

\newcommand{\rhoRef}[0]{\rho_{f,r}}
\newcommand{\gravity}[0]{\vec{g}}
\newcommand{\Hdiv}[1]{{H(\text{div}; #1)}}
\title{An efficient preconditioner for mixed-dimensional contact poromechanics based on the fixed stress splitting scheme}

\date{} 					
\author{\href{https://orcid.org/0009-0006-7095-3044}{\includegraphics[scale=0.06]{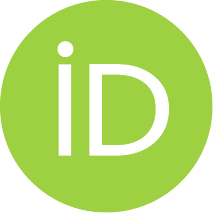}Yury Zabegaev}
\thanks{Center for Modeling of Coupled Subsurface Dynamics, Department of Mathematics, University of Bergen, Bergen, Norway}\\
\email{yury.zabegaev@uib.no}
\and
\href{https://orcid.org/0000-0002-0212-7959}{\includegraphics[scale=0.06]{orcid.pdf}Inga Berre}\footnotemark[1]\\
\email{inga.berre@uib.no}
\and
\href{https://orcid.org/0000-0002-0333-9507}{\includegraphics[scale=0.06]{orcid.pdf}Eirik Keilegavlen}\footnotemark[1]\\
\email{eirik.keilegavlen@uib.no}
\and
Kundan Kumar\footnotemark[1]\\
\email{kundan.kumar@uib.no}
}

\renewcommand{\shorttitle}{An efficient preconditioner for mixed-dimensional contact poromechanics}

\hypersetup{
pdftitle={An efficient preconditioner for mixed-dimensional contact poromechanics based on the fixed stress splitting scheme},
pdfsubject={An efficient preconditioner for mixed-dimensional contact poromechanics based on the fixed stress splitting scheme},
pdfauthor={Yury Zabegaev, Eirik Keilegavlen, Inga Berre, Kundan Kumar},
pdfkeywords={Poromechanics, Linear solver, Preconditioner, Fixed stress, Mixed-dimensional, Contact mechanics},
}
\begin{document}
\maketitle
\begin{abstract}
Numerical simulation of fracture contact poromechanics is essential for various applications, including CO\textsubscript{2} sequestration, geothermal energy production and underground gas storage. 
Modeling this problem accurately presents significant challenges due to the complex physics involved in strongly coupled poromechanics and frictional contact mechanics of fractures. The robustness and efficiency of the simulation heavily depends on a preconditioner for the linear solver, which addresses the Jacobian matrices arising from Newton’s method in fully implicit time-stepping schemes. Developing an effective preconditioner is difficult because it must decouple three interdependent subproblems: momentum balance, fluid mass balance, and contact mechanics. The challenge is further compounded by the saddle-point structure of the contact mechanics problem, a result of the Augmented Lagrange formulation, which hinders the direct application of the well-established fixed stress approximation to decouple the poromechanics subproblem. In this work, we propose a preconditioner 
that combines nested Schur complement approximations with a linear transformation, which addresses the singular nature of the contact mechanics subproblem.
This approach extends the fixed stress scheme to both the matrix and fracture subdomains. We investigate analytically how the contact mechanics subproblem affects the convergence of the proposed fixed stress-based iterative scheme and demonstrate how it can be translated into a practical preconditioner. The scalability and robustness of the method are validated through a series of numerical experiments.
\end{abstract}


\section{Introduction}

Fractures can significantly impact various subsurface operations, including
CO\textsubscript{2} sequestration \cite{FAN20191054,LIU20191,flemisch2024fluidflower}, geothermal energy production \cite{PAN201919,WEI2019120,ASAI2019763},
and underground gas storage \cite{zhou2019seismological,karev2019geomechanical,FIRME2019103006}, by influencing fluid flow and deformation processes within the subsurface. Fractures can serve as preferential pathways for fluid flow, impacting fluid pressure and transport in the surrounding rock matrix. In addition, changes in fluid pressure affect the mechanical stress of rock formations. This can cause frictional sliding and opening of fractures, threatening containment and increasing the risk of unacceptable levels of induced seismicity \cite{ellsworth2013injection}. 

To understand these processes, numerical simulations of fracture contact poromechanics, which can model slip and opening of fractures in deformable porous media, are necessary. However, accurate numerical modeling of fracture contact poromechanics poses a significant challenge due to the problem's geometric structure and complex physics involved. A model must account for fluid and friction forces on fracture surfaces and satisfy conditions for fracture sliding and opening.
Coulomb friction combined with the non-penetration condition, in particular, is notoriously difficult to model \cite{kikuchi1988contact,matrins1987,sofonea2012mathematical} because of its non-smooth transitions between regions of different fracture states: stick, slide, or open -- leading to changes in the governing constitutive laws. 
The challenge is further intensified by the coupling with 
fluid flow in the fractures and the Biot poromechanics model for the matrix surrounding the fractures \cite{coussy2004poromechanics}. This integrates fluid flow, matrix deformation and fracture contact mechanics, leading to a highly coupled and nonlinear model that is difficult to solve numerically.
To address these challenges, we follow the explicit fracture representation approach \cite{berre2019flow,porepy2024,berge_finite_2020}, which resolves the fractures and their intersections as lower-dimensional geometrical objects embedded in the computational domain. 
We reformulate the inequalities governing the contact mechanics relations as complementarity equations using an Augmented Lagrange method \cite{HUEBER20053147,wriggers2005}. The strong multiphysics couplings advocate for a fully implicit time discretization and applying Newton's method to solve the discretized problem. Consequently, the problem's complexity is propagated to the linear solver level, where an efficient and robust preconditioner for the fully coupled system is needed.
Both contact mechanics and poromechanics problems, when considered independently, have well-established methods for solving linear systems. The contact mechanics equations in the Augmented Lagrange formulation result in a saddle-point structure. 
Numerous linear solvers have been developed for saddle-point problem in various applications \cite{bacq_allatonce_2023,notay_new_2016,benzi_numerical_2005}, including the contact mechanics \cite{voet_internodes_2022,wiesner_algebraic_2021} and related to fractures in the subsurface \cite{franceschini_scalable_2022,franceschini_reverse_2022,franceschini_block_2019}.
Similarly, poromechanics has well-established preconditioning methods, primarily based on the fixed stress approximation \cite{white_block-partitioned_2016,mikelic_convergence_2013,kim_stability_2011,storvik_optimization_2018,BOTH2017101}, which utilizes idea that the volumetric hydro-mechanical stress changes slowly in time. 
However, integrating the fixed stress method into a fractured model is far from straightforward. The presence of fluid flow in both the fractures and the surrounding matrix requires specific treatment at the linearized level \cite{hu_effective_2023}. Furthermore, combining the fixed stress approach with contact mechanics is particularly challenging because many of the well-established saddle-point methods assume a specific structure of the matrix in hand: For a saddle-point matrix $\left[\begin{smallmatrix}
    A & B \\ C & D
\end{smallmatrix}\right]$ with a singular submatrix $D$, the submatrix $A$ needs to represent some atomic operator, e.g. a Laplacian, to make the analysis of the method feasible \cite{benzi_numerical_2005}. In our case, the operator $A$ encapsulates the coupled poromechanical problem, which limits analysis and hinders the performance of these methods.
One more challenge of designing a preconditioner is that it should be created with scalability in mind, ensuring that it can efficiently handle large-scale problems typical in real-world applications, such as reservoir simulations with large fracture networks. This requires the algorithm to maintain computational efficiency as the problem size grows, while also ensuring robustness 
for realistic parameters.

In this paper, we develop a robust and scalable linear preconditioner for frictional contact poromechanics.
To deal with possible singularities in the contact conditions, we apply a linear transformation to the linearized saddle-point problem before solving it, similar to \cite{bacq_allatonce_2023,notay_new_2016}.
This approach is combined with the fixed stress-based decoupling of the poromechanics equations based on \cite{girault_convergence_2016,white_block-partitioned_2016} with special treatment of degrees of freedom related to contact mechanics and fluid flow between fractures and the surrounding matrix, similar to \cite{hu_effective_2023}. Following this, we describe the practical algorithm for preconditioning the system, which decouples the problem into suitable blocks for an algebraic multigrid (AMG) method and ensures its scalability. Next, we study how the fixed stress approximation applies to our problem and provide convergence guarantees under certain simplifications.
We validate the preconditioner's performance through numerical experiments, demonstrating its robustness and scalability.

The structure of the manuscript is as follows: \Cref{sec:model_description} describes the governing equations, including constitutive laws, boundary and initial conditions of the mixed-dimensional poromechanical model with fractures. \Cref{sec:discretization} focuses on the discretization of the model.
\Cref{sec:preconditioner} describes the preconditioner algorithm, which relies on the fixed stress approximation. The suitability of this approximation is studied in \Cref{sec:convergence_study}, where we recast the preconditioner as a sequential iterative scheme, prove its contraction under certain assumptions and show how this result transfers to the preconditioner. In \Cref{sec:numerical_experiments}, we demonstrate the results of numerical experiments on the performance of the preconditioner. The concluding remarks are given in \Cref{sec:results}.
\section{Model description}
\label{sec:model_description}
The current section describes the mathematical model of mixed-dimensional flow and deformation of a poroelastic medium with frictional fracture deformation governed by contact poromechanics. The model has been introduced previously \cite{berge_finite_2020,porepy2021,porepy2024,stefansson_fully_2021}; this is therefore a brief review, following the presentation in \cite{porepy2024}. We first review the mixed-dimensional model geometry, before introducing the governing equations, which consist of i) conservation laws for fluid mass and momentum, ii) contact mechanics relations, specifically a non-penetration condition and the Coulomb friction law, iii) constitutive laws, and iv) boundary and initial conditions.
\subsection{Mixed-dimensional geometry}
\label{sec:mix_dim_geo}
The medium is described as a collection of subdomains $\Omega_i$ of different dimensions $d_i$, with $d_i \in \left\{0, ..., D \right\}$ and $D \in \left\{ 2, 3 \right\}$. In the case $D = 3$, the model can contain the following subdomains:
\begin{itemize}
    \item porous medium -- $3D$ subdomains;
    \item fractures -- $2D$ subdomains;
    \item intersections of fractures -- $1D$ subdomains;
    \item intersections of intersections -- $0D$ subdomains.
\end{itemize}
The width of the fracture $\Omega_i$ is characterized by its aperture $a_i$. To represent the spatial extension of $\Omega_i$ of dimension $d_i < D$ orthogonal to $\Omega_i$, we define the specific volume $\nu_i = a_i ^ {D - d_i}$. For fractures $\nu_i$ is equal to the aperture $a_i$, and for fracture intersections it is equal to $a_i^2$.
For consistency, we let $\nu_i=1$ if $d_i=D$.
The coupling between neighboring pairs of subdomains with a difference in dimension equal to one is facilitated by an interface $\Gamma_j$. Where relevant, we identify the higher-dimensional and lower-dimensional neighbor subdomains of an interface by subscripts $h$ and $l$, respectively. We write $\partial \Omega_i$ for the boundary of $\Omega_i$, while the internal part geometrically coinciding with the interface $\Gamma_j$ is $\partial_j \Omega_i \subseteq \partial \Omega_i$.
The projection operator of relevant quantities from subdomain $\Omega_i$ to interface $\Gamma_j$ is denoted by $\Pi^i_j$ and the reverse operator by $\Xi^i_j$.
We use the subscripts to identify the subdomain or interfaces, where the quantities are defined, and the superscripts $f$ and $s$ to denote fluid and solid quantities, respectively. However, we suppress subscripts and superscripts if the context allows.
\subsection{Conservation laws}
The following subsection presents the mass and momentum conservation laws for the relevant subdomains.
The fluid mass conservation equation for the subdomains $\Omega_i$ of dimension $d_i \in \left\{ 0, ..., D \right\}$:
\begin{equation}
    \label{eq:mass_conservation}
    \frac{\partial}{\partial t} \left( \nu_i \rho_i^f \varphi_i \right) - \nabla \cdot \left( \nu_i \rho_i^f \mathbf{v}_i \right) -\sum_{j \in \hat{S}_i} \Xi^i_j \left( \nu_j \rho_j^f v_j \right) = \psi_i,
\end{equation}
\noindent
where $\rho_i^f$ and $\rho_j^f$ are the fluid densities defined in a subdomain and on an interface, respectively; $\varphi_i$ is the porosity; $\mathbf{v}_i$ and $v_j$ are the volumetric fluid fluxes in a subdomain and on an interface, respectively; $\psi_i$ is a source or sink of fluid mass; $\nu_j \coloneq \Pi_i^h \nu_h$ is the interface specific volume; and the set $\hat{S}_i$ contains all interfaces to higher-dimensional neighbors of $\Omega_i$. Fluxes to lower-dimensional neighbors are treated as internal boundary conditions. 
The second term is void for $d_i = 0$, as 0D domains allow no internal flux, 
while the third term is void for $d_i = D$ as these subdomains have no higher-dimensional neighbors.
The quasi-static momentum conservation equation is imposed on the matrix subdomain ($d_i = D$):
\begin{equation}
    \label{eq:momentum_conservation}
    -\nabla \cdot \sigma_i  = \mathbf{F}_i,
\end{equation}
with $\sigma_i$ being the total stress tensor and $\mathbf{F}_i$ representing body forces. 
\subsection{Contact mechanics relations}
In the following subsection, we consider a matrix-fracture pair $\Omega_h$ and $\Omega_l$ of dimensions $d_h = D$ and $d_l = D - 1$, respectively. The two interfaces on each side of $\Omega_l$ are denoted $\Gamma_j$ and $\Gamma_k$. 
We define a fracture normal vector $\mathbf{n}_l$ to coincide with $\mathbf{n}_h$ on the $j$-side of the fracture and introduce the fracture contact 
traction ${\lambdafull}_l$, defined according to the direction of $\mathbf{n}_l$. 
For a generic 3D vector $\mathbf{\iota}$, the tangential and normal components are denoted 
\begin{equation}
    \mathbf{\iota}_\mathrm{n} = \mathbf{\iota} \cdot \mathbf{n}_l, \quad \mathbf{\iota}_\tau = \mathbf{\iota} - \mathbf{\iota}_\mathrm{n}.
\end{equation}
We also introduce the jump in interface displacements across $\Omega_l$:
\begin{equation}
    \ujump = \Xi^l_k \disp_k - \Xi^l_j \disp_j.
\end{equation}
where $\Xi$ is the subdomain-to-interface projection operator defined in \Cref{sec:mix_dim_geo}.
We achieve balance between the traction on the two fracture surfaces by enforcing each of them to equal the total fracture traction:
By Newton's third law, there is balance between the contact traction and the total traction on each of the fracture surfaces:
\begin{equation}
\begin{aligned}
\label{eq:intf_force_balance}
        \Pi^h_j \sigma_h \cdot \mathbf{n}_h & = \Pi^l_j \left( \lambdafull_l - p_l \mathbf{I} \cdot \mathbf{n}_l \right),\\
            -\Pi^h_k \sigma_h \cdot \mathbf{n}_h &= \Pi^l_k \left( \lambdafull_l - p_l \mathbf{I} \cdot \mathbf{n}_l \right),
\end{aligned}
\end{equation}
\noindent
where $\mathbf{I}$ denotes the identity matrix.
We enforce the non-penetration condition of the fracture's surfaces by imposing (the subscript $l$ is omitted for readability)
\begin{equation}
\label{eq:contact_normal_ineq}
   \ujumpn - g \geq 0;
   \quad
    \lambdan \leq 0;
    \quad
    \lambdan \left( \ujumpn - g \right) = 0,
\end{equation}
where $g$ is the normal gap between the fracture's surfaces when they are in mechanical contact. The second inequality reflects that compressive normal contact traction corresponds to negative $\lambdan$ by the definition of $\lambdafull$. 
The friction bound, $b$, is defined as a Coulomb-type friction law with a constant friction coefficient $F$:
\begin{equation}
\label{eq:friction_bound}
    b = -F \lambdan.
\end{equation}
With sliding, the velocity of the tangential displacement jump is denoted $\ujumpt$, and the friction model is given by:
\begin{equation}
\begin{aligned}
\label{eq:contact_tangential_ineq}
    \norm{\lambdat} &\leq b \\
    \norm{\lambdat} &< b \quad \rightarrow \quad \ujumpt = 0, \\
    \norm{\lambdat} &= b \quad \rightarrow \quad \exists \zeta \in \mathbb{R}^+ : \ujumpt = \zeta \lambdat
\end{aligned}
\end{equation}
The three relations state that: (i) tangential stresses are bounded, (ii) tangential deformation occurs only if the bound is reached and (iii) tangential stresses and deformation increments are parallel and co-directed.
The contact mechanics inequalities prescribe to each point of the fracture one of three states:
\begin{itemize}
    \item Sticking -- the fracture sides are in mechanical contact and not moving relative to each other; the traction $\lambdafull$ is nonzero.
    \item Sliding -- the fracture sides are in mechanical contact with the tangential contact force $\lambdat$ having reached the friction bound, and the fracture sides are moving relative to each other.
    \item Open -- the fracture sides are separated from each other, no mechanical contact occurs; the traction $\lambdafull$ is zero.
\end{itemize}
The inequalities \cref{eq:contact_normal_ineq,eq:contact_tangential_ineq} can be reformulated as equalities $C_\text{n} = 0$ and $C_\tau = 0$ based on complementarity functions according to the augmented Lagrangian formulation, see \cite{HUEBER20053147,stefansson_fully_2021,berge_finite_2020}:
\begin{subequations}
\label{eq:contact_equalities}
    \begin{equation}
        C_\text{n}\coloneq \lambdan + \text{ max}\left\{ 0, -\lambdan - c (\ujumpn - g)  \right\};
    \end{equation}
    \begin{equation}
        C_\tau \coloneq \left( -\lambdat \text{max}\left\{ b, \norm{\lambdat + c\ujumpt} \right\} + \text{max}\left\{ b, 0 \right\} \left(\lambdat + c \ujumpt \right) \right) \left(1 - \chi \right) + \chi \lambdat,
    \end{equation}
\end{subequations}
where $c>0$ is a  numerical parameter, and $\chi$ is a characteristic function: $\chi = 1 \text{ if } |b| < \varepsilon$ and $\chi = 0 \text{ otherwise}$;
$\varepsilon$ is a numerical parameter that defines the fracture as open even for small positive values of the friction bound. This can reduce the sensitivity of the nonlinear solver to numerical errors and thereby improve solver convergence.

\subsection{Constitutive relations}
\label{sec:constitutive_laws}

The constitutive relations provided below are in large part based on \cite{coussy2004poromechanics}.
The volumetric fluid flux is modeled using Darcy's law:
\begin{equation}
\label{eq:darcy}
    \mathbf{v}_i = - \frac{\perm_i}{\visc} \left( \nabla p_i - \rho^f_i \gravity \right),
\end{equation}
where $\perm_i$ is the permeability tensor, $\visc$ denotes fluid viscosity and density, and $\gravity$ is the gravity vector. We assume $\perm$ is constant in the matrix, whereas the fracture permeability is given by \cite{zimmerman1996hydraulic}:
\begin{equation}
\label{eq:fracture_permeability}
    \perm_i = \frac{a_i^2}{12}\mathbf{I}, \quad d_i = D - 1.
\end{equation}
Note that $a_i$ depends on $\ujump$ as detailed in Equation \eqref{eq:aperture}. The permeability of intersections is computed as the average of the permeabilities in the intersecting fractures.
The solid density is constant, while the fluid density is given by:
\begin{equation}
    \rho^f = \rhoRef \exp{\left(c_f (p - p^0)\right)},
\end{equation}
where $c_f$ denotes compressibility, $\rhoRef$ and $p^0$ denote the reference density and pressure, respectively. The viscosity is constant.
The volumetric interface flux is proportional to the pressure jump across $\Gamma_j$ via a Darcy-type law \cite{Martin2005}:
\begin{equation}
\label{eq:interface_flow}
    v_j = - \frac{\mathcal{K}_j}{\mu_{f,j}} \left[ \frac{2}{\Pi^l_j a_l} \left( \Pi^l_j p_l - \Pi^h_j p_h \right) -\mathbf{g} \cdot \mathbf{n}_h \rho_j^f \right],
\end{equation}
with $\mathcal{K}_j$ and $\mu_{f,j}$ denoting the interface permeability and viscosity inherited from the lower-dimensional neighbor subdomain. The factor $\frac{2}{\Pi^l_j a_l}$ represents half the normal distance across the fracture. For an advective quantity $\xi_j$ representing $\rho_j$ and $\mu_{f,j}$, we use an inter-dimensional upwinding based on $v_j$:
\begin{equation}
    \xi_j = \begin{cases}
        \Pi^h_j \xi_h \quad &\text{if } v_j > 0, \\
        \Pi^l_j \xi_l \quad &\text{otherwise}.
    \end{cases}
\end{equation}
The total poromechanical stress tensor is given by Hooke's law and accounts for fluid pressure:
\begin{equation}
\label{eq:poromech_stress}
    \sigma_i = \sigma^0_i + G \left( \nabla \disp_i + \nabla 
    \disp_i^T \right)
    + \Lame 
    \text{tr}(\nabla \disp_i) 
    \mathbf{I} - \alpha \left( p_i - p^0_i \right) \mathbf{I},
\end{equation}
where $G$ and $\Lame$ are the Lame coefficients, $\alpha$ is the Biot coefficient and $\text{tr}(\cdot)$ denotes the trace of a matrix.
The matrix porosity depends on pressure and displacement according to \cite{coussy2004poromechanics}:
\begin{equation}
    \varphi = \varphi^0 + \alpha \nabla \cdot \mathbf{u} + \frac{(\alpha - \varphi^0)(1 - \alpha)}{\Lame + \frac{2}{3} G} \left( p - p^0 \right),
\end{equation}
We set unitary fracture and intersection porosity.
Fracture roughness effects are incorporated through the gap function:
\begin{equation}
\label{eq:gap_function}
    g = g^0 + \frac{\Delta u_\text{max} {\lambdan}}{\Delta u_\text{max} K_n - {\lambdan}} + \norm{\disp_\tau} \tan \theta,
\end{equation}
where $g^0$ is the steady-state gap, $\Delta u_\text{max}$ is the maximum normal fracture closure, $K_n$ is the fracture normal stiffness and $\theta$ is the shear dilation angle. The second term of the expression corresponds to the normal elastic deformation of the fracture according to the Barton-Bandis model \cite{barton_strength_1985}. The third term corresponds to the shear dilation. 

The aperture of fracture subdomains is defined as follows:
\begin{equation}
\label{eq:aperture}
    a = a^0 + \ujumpn.
\end{equation}
Here, $a^0$ denotes a residual hydraulic aperture. In intersection subdomains (i.e., $d_i < D - 1$), we compute $a_i$ from the mean aperture of higher-dimensional neighbors:
\begin{equation}
    a_i = \frac{1}{|\hat{S}_i|}\sum_{j \in \hat{S}_i} \Xi^i_j \Pi^h_j a_h.
\end{equation}
\subsection{Initial and boundary conditions}
To close the system of equations, we provide the initial conditions, as well as boundary conditions on the internal and external boundaries of each subdomain. On internal boundaries $\partial_j \Omega_i$, we require continuity of normal mass fluxes and displacements (for $d_i = D$):
\begin{equation}
    \nu_i \rho_i^f \mathbf{v}_i \cdot \vec{n}_h = \Xi^i_j \nu_j \rho^f_j v_j;
    \quad
    \mathbf{u}_i = \Xi^i_j \mathbf{u_j}.
\end{equation}
On immersed fracture tips, we require the mass flux to equate to zero. On external boundaries, Dirichlet or Neumann conditions are set for the mass and momentum conservation equations.

\section{Discretization}
\label{sec:discretization}
As the discretization of the governing equations is not a main point of this work, we briefly summarize the main ingredients; for more information see \cite{porepy2021,porepy2024}.
The computational simplex grid is constructed to conform to immersed lower-dimensional geometrical objects (fractures and intersections) in the sense that each lower-dimensional subdomain coincides geometrically with a set of faces on the surrounding higher-dimensional grid.
The spatial discretization is based on finite volume methods. The fluid fluxes are discretized using the multi-point flux approximation method \cite{Aavatsmark2002}, whereas the 
poromechanical stress is discretized using the multi-point stress approximation (MPSA) method \cite{Nordbotten2016}. Advective terms are discretized by standard upwinding in single-phase flow. The implicit Euler scheme is used for temporal discretization. 
The nonlinear system, which consists of discretized versions of \cref{eq:mass_conservation,eq:momentum_conservation,eq:interface_flow,eq:intf_force_balance,eq:contact_equalities}, is only semi-smooth owing to the contact mechanics equations.
We solve the system using a semi-smooth Newton method \cite{ito2003semi}.
The core of the method is to invert the generalized Jacobian matrix, which matches the original Jacobian in the smooth regions and takes the derivative value from one or the other side in the non-smooth regions \cite{ito2003semi,berge_finite_2020}. We will refer to it as just the Jacobian, $J$.
\section{Preconditioner for mixed-dimensional contact poromechanics}
\label{sec:preconditioner}
We start constructing a preconditioning algorithm to address linear systems that arise in linearization of \cref{eq:mass_conservation,eq:momentum_conservation,eq:interface_flow,eq:intf_force_balance,eq:contact_equalities} within the semi-smooth Newton's method steps.
To begin with, we recall block elimination procedure based on the block $\mathcal{LDU}$ decomposition of a $2\times2$ block matrix:
\begin{equation}
J=
    \begin{bmatrix}
        A & B \\ C & D
    \end{bmatrix}
    = \mathcal{L D  U} =
    \begin{bmatrix}
        I              &  \\
        C  A^{-1} & I
    \end{bmatrix}
    \begin{bmatrix}
        A & \\
          & S_A
    \end{bmatrix}
    \begin{bmatrix}
        I & A^{-1} B \\
          & I
    \end{bmatrix},
\end{equation}
where $S_A$ is the Schur complement: $S_A \coloneq D - C A^{-1} B$.
Using the upper-triangular matrix $(\mathcal{DU})^{-1}$ as a right preconditioner, we get 
$\mathcal{LDU}(\mathcal{DU})^{-1}=\mathcal{L}$. As the diagonal elements of $\mathcal{L}$ all equal one, so do its eigenvalues, thus a Krylov subspace method combined with this preconditioner will converge rapidly. In practice, a sparse approximation of the Schur complement is needed, and the linear systems based on $A$ and $S_A$ are solved inexactly.

The choice of the Schur complement approximation and the inexact subsolvers depends on the underlying problem. For matrices with more than two block rows and columns, a recursive $\mathcal{LDU}$ decomposition applies, and the corresponding block-triangular preconditioner is a suitable option. In this case, nested Schur complements must be approximated. For our problem, the Jacobian, schematically represented in \cref{fig:elimination_order} (left), has the following features: The contact mechanics submatrix is singular and corresponds to the saddle-point structure of the matrix. It is strongly coupled with the elasticity submatrix, which, in turn, is strongly coupled with the fluid mass balance submatrix. 
In the fracture-less poromechanics, where the matrices $A$ and $D$ correspond to the subproblems of momentum and mass balance, respectively, 
it is common to use the fixed stress approximation for the Schur complement reduction, often combined with an AMG method to solve the subproblems \cite{white_block-partitioned_2016}.
Unfortunately, this approach cannot readily be extended to poromechanics with fractures as the inclusion of fracture contact mechanics (for the momentum conservation equation) and fluid flow between subdomains (for the mass balance equation) may reduce the ellipticity of the corresponding subproblems, and thereby jeopardize AMG performance.

This section presents an extension of the fixed stress approximation based on a nested Schur complement decomposition. The algorithm consists of two stages: preprocessing and application. In the preprocessing stage, a right linear transformation is applied to the Jacobian to prevent the contact mechanics submatrix from being singular; the details follow in \Cref{sec:linear_transformation}. The application stage consists of three nested Schur complement decompositions, $S^{i}$ with the following structure:
\begin{enumerate}
    \item $S^1$ decouples the contact mechanics equations according to \Cref{sec:eliminating_contact_mechanics};
    \item $S^2$ applies a fixed stress-based approximation to decouple the momentum balance equation according to \Cref{sec:eliminating_force_balance};
    \item $S^3$ decouples the interface fluid flow according to \Cref{sec:eliminating_intf_flow}.
\end{enumerate}
\cref{fig:elimination_order} illustrates the algorithm and the indexing scheme used to represent submatrices of the Jacobian, which we will refer to as $J_{ij}$ where $i, j \in \{1, 2, 3, 4, 5\}$, e.g. $J_{11}$ corresponds to the contact mechanics submatrix.

\begin{figure}[htbp]
\label{fig:elimination_order}
    \centering
    \includegraphics[width=1\linewidth]{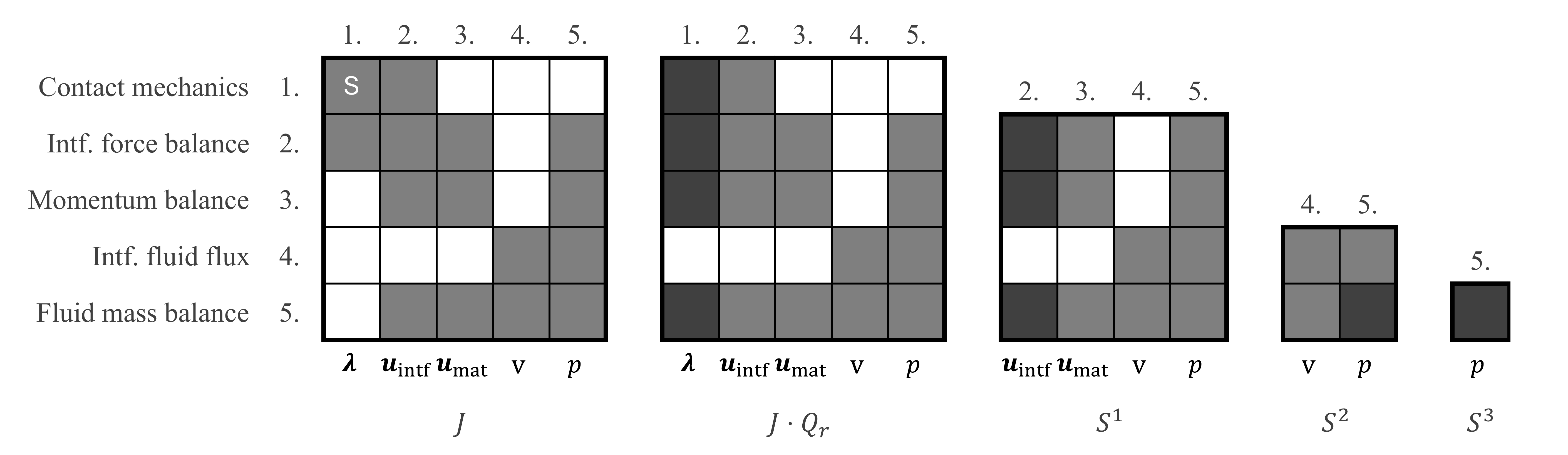}
    \caption{The block structure of the matrices on different stages of the preconditioner algorithm. White cells represent empty submatrices, light-gray cells represent non-empty submatrices of the original Jacobian. Dark-gray cells represent modified submatrices. The ``S'' letter corresponds to the singular submatrix. From left to right: 1 -- the original Jacobian, 2 -- the Jacobian after the linear transformation (see \Cref{sec:linear_transformation}), 3 -- the first-level Schur complement (see \Cref{sec:eliminating_contact_mechanics}), 4 -- the second-level Schur complement (see \Cref{sec:eliminating_force_balance}), 5 -- the third-level Schur complement (see \Cref{sec:eliminating_intf_flow}).
    }
\end{figure}

\subsection{Preprocessing}
\label{sec:linear_transformation}
$J_{11}$ consists of block-diagonal submatrices of size $D \times D$, each of which corresponds to a specific fracture cell.
Unfortunately, the submatrices that correspond to the sticking fracture state are singular, as can be seen from \cref{eq:contact_normal_ineq,eq:contact_tangential_ineq}, where the contact relations do not depend on $\lambdat$. 
To make the submatrix $J_{11}$ invertible, we note that the contact mechanics equations are coupled only with the interface force balance equation; the Jacobian matrix can be represented as 
\begin{equation}
\label{eq:jac_contact}
    J = \begin{bmatrix}
        J_{11} & J_{12} & 0 \\
        J_{21} & J_{22} & J_{2 \star} \\
        0               & J_{\star 2} & J_{\star \star} \\
    \end{bmatrix},
\end{equation}
where the symbol $\star$ incorporates the remainder, the submatrices with the indices $i, j \in \{3, 4, 5\}$.
This motivates a linear transformation matrix, inspired by  \cite{notay_new_2016,bacq_allatonce_2023}:
\begin{equation}
    \label{eq:Qright}
    Q_r = \begin{bmatrix}
        I_{11}& 0& 0& \\
        -{D_{22}}^{-1}  J_{21}& I_{22}& 0& \\
        0& 0& I_{\star \star}& \\
    \end{bmatrix},
\end{equation}
where $I_{ii}$ is an identity matrix of a shape that match to block structure of the Jacobian and
the subscript $r$ is used to indicate that $Q_r$ is applied from the right side.
${D_{22}}$ is a block-diagonal approximation of the submatrix $J_{22}$, with block size $D$ (the spatial dimension), hence 
${D_{22}}$ can be inexpensively inverted.
We apply $Q_r$ to the original matrix \cref{eq:jac_contact} and obtain the linearly-transformed system: $J Q_r  Q_r^{-1} x = r$,

where $x$ and $r$ are the unknown and the right-hand side vectors, respectively. We denote $\tilde{J} \coloneq J  Q_r$ and $\tilde{x} \coloneq Q_r^{-1} x$ and solve the transformed linear system $\tilde{J}  \tilde{x} = r$, the transformed matrix is shown second from the left on \cref{fig:elimination_order}. Note that the transformation is applied only once, in the preprocessing phase. 
The linear transformation changes the sparsity structure of the Jacobian to
\begin{equation}
\label{eq:J_tilde}
    \tilde{J} = \left[\begin{array}{c;{4pt/2pt}cc}
        J_{11} - J_{12}  D_{22}^{-1}  J_{21} & J_{12} & 0 \\   \hdashline[4pt/2pt]
        E_{21} & J_{22} & J_{2 \star} \\
        -J_{\star 2}  D_{22}^{-1}  J_{21} & J_{\star 2} & J_{\star \star}
    \end{array}\right],
\end{equation}
where $E_{21} \coloneq \left( I_{22} - J_{22}  D_{22}^{-1} \right)  J_{21}$ is an error matrix caused by the approximated inverse of $J_{22}$.

We also experimented with using a left transformation matrix $Q_l$ applied to the system $Q_l  J  x = Q_l  r$, with no notable difference in performance and convergence.
\subsection{First-level Schur complement: Eliminating contact mechanics}
\label{sec:eliminating_contact_mechanics}
The submatrix $\tilde{J}_{11} \coloneq J_{11} - J_{12}  D_{22}^{-1}  J_{21}$ is non-singular,  consisting of $D\times D$ block-diagonal submatrices, reflecting the local character of the contact conditions. Exact inversion of $\tilde{J}_{11}$ yields the first-level Schur complement
\begin{equation}
\label{eq:S^1_*}
    S^{1} = \begin{bmatrix}
        J_{22} + E_{22} & J_{2 \star} \\
        J_{\star 2} - \tilde{J}_{\star 1}  \tilde{J}_{11}^{-1}  J_{12} & J_{\star \star}\\
    \end{bmatrix},
\end{equation}
where $E_{22} \coloneq - E_{21} \tilde{J}_{11}^{-1} J_{12}$. $S^1$ is depicted as matrix 3 in \cref{fig:elimination_order}.
Note that $J_{\star \star}$ remains unmodified.
We observe that the linear transformation $Q_r$ and the efficient and inexpensive elimination of the contact mechanics degrees of freedom is applicable to any discretization of the contact (poro)mechanics problem which adopts Lagrange multiplier-based approaches. 
\subsection{Second-level Schur complement: Eliminating the momentum balance equations}
\label{sec:eliminating_force_balance}
Having eliminated the contact relations, we now turn to the elimination of the momentum balance equations, which is handled by the fixed stress scheme.
The first-level Schur complement \eqref{eq:S^1_*} can be rewritten as follows:
\begin{equation}
\label{eq:S^1}
    S^1 = \left[\begin{array}{cc;{4pt/2pt}cc}
        S^1_{22} & J_{23} &        & J_{25} \\
        S^1_{32} & J_{33} &        & J_{35} \\ \hdashline[4pt/2pt]
                 &        & J_{44} & J_{45} \\
        S^1_{52} & J_{53} & J_{54} & J_{55} \\
    \end{array}\right],
\end{equation}
where $S^1_{ij}$ denotes the corresponding part of $S^1$, specifically, $S^1_{22} = J_{22} + E_{22}$ and $S^1_{i2} = J_{i2} - \tilde{J}_{i1} \cdot \tilde{J}_{11}^{-1} \cdot J_{12}$, $i \in \{ 3, 4, 5\}$. 
It can be seen from \eqref{eq:S^1} that the nontrivial Schur complement appears only for $J_{55}$, as there is no coupling between mechanics and the interface fluid flow.
The coupled operator of the momentum balance and the interface force balance equation, which corresponds to indices \{2, 3\} is solved monolithically with AMG.
We utilize the fixed stress-based approximation to construct the second-level Schur complement $S^2_{55} = J_{55} + D_{55}$,
where $D_{55}$ is a diagonal matrix of the corresponding shape. Since the mass balance equation is defined in the ambient dimension, the fractures, and lower-dimensional geometrical subdomains, the coefficients of $D_{55}$ are different for the corresponding degrees of freedom. For the ambient dimension and for fracture subdomains, the following coefficients $l_\text{mat}$ and $l_\text{frac}$ are used, respectively:
\begin{equation}
    \label{eq:fixed_stress_coefs}
    l_\text{mat} = \dfrac{\alpha^2}{2 G / D + \Lame};
    \quad
    l_\text{frac} = \dfrac{\ujumpn \alpha^2 c_f}{\Lame (1/M + \varphi c_f)},
\end{equation}
where $\dfrac{1}{M} = \dfrac{\partial \varphi}{\partial p} = \dfrac{(\alpha - \varphi^0)(1 - \alpha)}{\Lame + \frac{2}{3}G}$. The first is the classical fixed stress coefficient \cite{mikelic_convergence_2013,kim_stability_2011,white_block-partitioned_2016} and the latter is derived in \cite{girault_convergence_2016} for the fractured poromechanics problem without contact mechanics. Its applicability to fractured contact poromechanics is analyzed in the next section.
Since the intersection subdomains are not coupled with the momentum balance equation, no stabilization is needed for them.
An alternative, fully algebraic, stabilization for fracture-related degrees of freedom is studied in \cite{franceschini_scalable_2022}.
In our experience, the two approaches yield insignificant difference in solver performance,
though the algebraic approach requires more time to construct and occupies additional non-diagonal entries of the matrix.

\subsection{Third-level Schur complement: Eliminating the interface flow}
\label{sec:eliminating_intf_flow}
The second-level Schur complement is expressed as follows:
\begin{equation}
\label{eq:S^2}
    S^2 = \left[\begin{array}{c;{4pt/2pt}c}
         J_{44}& J_{45}  \\ \hdashline[4pt/2pt]
         J_{54}& S^2_{55} 
    \end{array}\right].
\end{equation}
As shown in \cite{hu_effective_2023}, the interface fluid flow submatrix $J_{44}$ is diagonally dominant and $J_{44}$ can be approximated by its diagonal to yield the third-level Schur complement: 
\begin{equation}
    \label{eq:S^3}
    S^3 = S^2_{55} - J_{54} \cdot \text{diag}(J_{44})^{-1} \cdot J_{45}.
\end{equation}
To solve the linear systems based on $J_{44}$, we apply the ILU(0) algorithm.
The obtained third-level Schur complement contains the fluid mass balance degrees of freedom and is solved monolithically with AMG.
\subsection{Preconditioner algorithm summary}
We provide \cref{algorithm:construction,algorithm:application} of constructing and applying the preconditioner, where $\text{bdiag}(\cdot)$ denotes the $2\times2$ or $3\times3$ block diagonal approximation, and AMG$(\cdot, \cdot)$ and ILU0$(\cdot, \cdot)$ denote the application of the corresponding algorithms.
\begin{algorithm}[htbp]
\caption{Preconditioner $P$ construction.}
\textbf{Input:} $J_{ij}$, $i, j \in \{1,2,3,4,5\}$;
\begin{algorithmic}[1]
    \label{algorithm:construction}
    \State $D_{22}^{-1} = \text{bdiag}(J_{22})^{-1}$;
    \State Assemble $Q_r$ using $D_{22}^{-1}$ according to \cref{eq:Qright};
    \State $\tilde{J} = J \cdot Q_r$;
    \State $S^1_{22} = J_{22} - \tilde{J}_{21} \cdot \tilde{J}_{11}^{-1} \cdot J_{12}$;
    \State $S^1_{32} = J_{32} - \tilde{J}_{31} \cdot \tilde{J}_{11}^{-1} \cdot J_{12}$;
    \State Assemble the fixed stress submatrix $D_{55}$  according to \Cref{sec:eliminating_force_balance};
    \State $S^2_{55} = J_{55} + D_{55}$
    \State $S^3_{55} = S^2_{55} - J_{54} \cdot \text{diag}(J_{44})^{-1} \cdot J_{45}$.
\end{algorithmic}
\end{algorithm}
\begin{algorithm}[htbp]
\caption{Preconditioner $P$ application.}
\textbf{Input:} $\vec{w} = [\vec{w}_1, \vec{w}_2, \vec{w}_3, \vec{w}_4, \vec{w}_5]$;
\begin{algorithmic}[1]
    \label{algorithm:application}
    \State $\vec{v_5} = \text{AMG}(S_{55}^3, \vec{w}_5)$;
    \Comment{Fluid mass balance}
    \State $\vec{v_4} = \text{ILU0}(J_{44}, \vec{w_4} - J_{45} \cdot \vec{v_5})$;
    \Comment{Interface fluid flow}
    \State $
    \begin{bmatrix}
        \vec{v_2} \\ \vec{v_3} 
    \end{bmatrix} = \text{AMG}(
    \begin{bmatrix}
        S_{22}^1 & J_{23} \\ S_{32}^1 & J_{33}
    \end{bmatrix},
    \begin{bmatrix}
        \vec{w}_2 - J_{25} \cdot \vec{v}_5 \\ \vec{w}_3 - J_{35} \cdot \vec{v}_5
    \end{bmatrix}$);
    \Comment{Momentum and interface force}
    \State $\vec{v}_1 = \tilde{J}_{11}^{-1} \cdot (\vec{w}_1 - J_{12} \cdot \vec{v}_2)$; \Comment{Contact mechanics}
\end{algorithmic}
\textbf{Output:} $[\vec{v}_1, \vec{v}_2, \vec{v}_3, \vec{v}_4, \vec{v}_5]$.
\end{algorithm}
The proposed preconditioner approximates the upper triangular matrix of the block $\mathcal{LDU}$ decomposition of the linearly transformed Jacobian. In exact arithmetic, the GMRES method with this preconditioner should converge in two iterations. In practice, the quality of the preconditioner depends on the accuracy of the performance of the AMG subsolvers for the elasticity and the fluid mass balance submatrices, and on the performance of the ILU subsolver for the interface fluid flow. Also important is the quality of the matrix inverse approximations to build the corresponding Schur complements, namely, the block-diagonal inverse of $J_{44}$ and the fixed stress approximation. Applicability of the latter for the fractured contact poromechanics problem and for the linearly transformed Jacobian is addressed in the next section.

\section{Analysis of fixed stress stabilization for mixed-dimensional contact poromechanics}
\label{sec:convergence_study}
While the practical performance of the preconditioner developed in the previous section must be established by numerical experiments, it is of interest to first give a theoretical justification of its properties.
To that end, we study the sequential coupling between the mechanics and flow problems, and derive necessary conditions for the contraction properties, hence convergence, of the coupling scheme.
This approach has previously been applied to study the fixed stress approach for non-fractured porous media, e.g. \cite{mikelic_convergence_2013,storvik_optimization_2018}, where the theoretically derived conditions describe well also the practical performance of the scheme.
An analysis for a fixed stress-like scheme for fractured porous media, though with a model that allowed for interpenetration of the fracture surfaces and without friction,
was presented in \cite{girault_convergence_2016}. The results presented below can be considered an extension of that work.
Our analysis is carried out on the continuous level, but we discuss the connection between the sequential scheme and the preconditioner developed in \Cref{sec:preconditioner} towards the end of this section.
To make the analysis feasible, we introduce the following simplifications of the model introduced in \Cref{sec:model_description}:

the permeability in fractures and matrix are taken as constant. 
The reference values of stress, pressure, and displacement are taken as zero.
The fracture contact state is considered known due to linearization but it may vary in space.
Finally, we consider a domain with a single fracture.

The steps for the contraction proof are as follows \cite{girault_convergence_2016}: We first introduce the sequential coupling scheme in a weak formulation. For this scheme, we take the difference of two successive iterations of the scheme and obtain energy estimates for the flow model including the mass balance and Darcy flux equations. Next, we obtain estimates for the momentum balance equation taking into the account the contact mechanics term. After that, we consider the linearized contact term and estimate it in terms of displacement. Finally, we combine the above estimates, demonstrate the contraction and show how the contact term affects the convergence condition.

\subsection{Fixed stress sequential iterative scheme}
We consider the domain which includes a single fracture and the ambient dimension $\Omega = \Omega_h \cup \Omega_l$. It is convenient to introduce an auxiliary partition of $\Omega$ into two non-overlapping subdomains $\Omega^+$ and $\Omega^-$. We introduce the virtually extended fracture $\Omega_l^\text{ext}$, which is a Lipschitz surface containing the original fracture: $\Omega_l \subset \Omega_l^\text{ext}$. To simplify the discussion, we use a superscript $\star$ to denote either $+$ or $-$. $\Omega^\star$ is adjacent to $\Cfrac^\star$. For any function $f$ defined in $\Omega$, we extend the star notation to $\Omega^\star$ and set $f^\star = f_{\vert \Omega^\star}$.
Let ${\Omega_l^\text{ext}}^\star = \partial \Omega^\star \backslash \Omega_l^\text{ext}$.
Let $W = H^1(\Omega^+ \cup \Omega^-)$ with the norm $\norm{v}^2_W = \norm{v^+}^2_{H^1(\Omega^+)} + \norm{v^-}^2_{H^1(\Omega^-)}$.
The space for the displacement is $L^\infty (0, T; V)$, where $V$ is a closed subspace of $H^1(\Omega_h)^D$ with the corresponding norm:
\begin{equation}
V = \{ \vec{v} \in W^D; \jump{\vec{v}}_{{\Omega_l^\text{ext}} \backslash \Cfrac} = \vec{0}, \vec{v}^\star_{\vert {\Omega_l^\text{ext}}^\star} = 0, \star= +, - \};
\quad
\norm{\vec{v}}_V^2 = \left( \Sigma^D_{i=1} \norm{v_i}^2_W \right),
\end{equation}
the subscript $D$ refers to the ambient dimension (2 or 3).

The space for the Darcy velocity in the matrix is:
\begin{equation}
    \mathbf{Z} = \left\{
    \vec{q} \in \Hdiv{\Omega^+ \cup \Omega^-}; \quad \jump{\vec{q}} \cdot \vec{n}^+ = 0 \text{ on } {\Omega_l^\text{ext}} \backslash \Cfrac, \quad \vec{q} \cdot \vec{n} = 0 \text{ on } \partial \Omega
    \right\},
\end{equation}
with norm given by:

\begin{equation}
    \norm{\vec{q}}_\mathbf{Z}^2 = 
    \norm{\vec{q}}^2_\Hdiv{\Omega^+} + \norm{\vec{q}}^2_\Hdiv{\Omega^-}.
\end{equation}
The space for the Darcy velocity in the fracture is:
\begin{equation}
    \mathbf{Z}_\Cfrac = \left\{
    \vec{q}_c \in L^2(\Cfrac)^{D - 1}; \overline{\nabla} \cdot \vec{q}_c \in H^{-\frac{1}{2}}(\Cfrac)
    \right\},
\end{equation}
where $\overline{\nabla} \cdot$ is the surface gradient on the fracture.
The corresponding norm is defined by:
\begin{equation}
    \norm{\vec{q}_c}_{\mathbf{Z}_\Cfrac}^2 = 
    \norm{\vec{q}_c}^2_{L^2(\Cfrac)}
    +
    \norm{\overline{\nabla} \cdot \vec{q}_c}^2_{ H^{-\frac{1}{2}}(\Cfrac)}.
\end{equation}

We denote the scalar products according to the norms in $\Omega_h$ and $\Cfrac$ by $(\cdot, \cdot)$ and $(\cdot, \cdot)_\Cfrac$, respectively.

The space for the contact traction $\lambdafull$ is $H^{-\frac{1}{2}}(\Cfrac^{D})$ in the continuous setting. In the discrete formulation, its regularity is improved so that $\lambdafull \in L^2(\Cfrac^D)$. In this study, we assume that the discretized problem is well-posed with a unique solution and focus on the convergence of the fixed stress splitting scheme.
We define auxiliary coefficients that will appear later in the fixed stress scheme:
\begin{equation}
    \beta = \dfrac{1}{M \alpha^2} + \dfrac{c_f}{\alpha^2}\varphi_0 + \dfrac{1}{\Lame};
    \quad 
    \chi = \left\{ \dfrac{c_{fc}}{\Lame^2 \beta - \Lame}\right\}^\frac{1}{2};
\end{equation}
\begin{equation}
    \label{eq:aux_coefs_gamma_c}
    c_{fc} = a_0 c_f;
    \quad
    \gamma_c = \dfrac{c_{fc}}{\Lame \beta - 1};
    \quad 
    \beta_c = c_{fc} + \gamma_c,
\end{equation}
where $\varphi_0$ and $a_0$ are the reference porosity and aperture, respectively, and $1/M$ is defined under \cref{eq:fixed_stress_coefs}.
The following denotes the volumetric mean stress in the matrix:
\begin{equation}
    \sigma_v \coloneq \sigma_{v,0} + \Lame \nabla \cdot \disp - \alpha \left( p - p_0\right),
\end{equation}
where $\sigma_{v,0}$ is the initial volumetric stress. As $\sigma_{v,0}$ and $p_0$ are constants in time, we have:
$\partial_t \sigma_v = \alpha \partial_t p - \Lame \nabla \cdot \partial_t \disp$.
The analogous term is defined in fractures: $\chi \partial_t \sigma_f = \gamma_c \partial_t p_c - \partial_t \ujumpn$. 
The Augmented Lagrange contact mechanics \cref{eq:contact_equalities} can be reformulated as follows:
\begin{subequations}
\label{eq:contact_equalities_for_fs}
    \begin{align}
    C_\text{n} = \begin{cases}
    - c (\ujumpn - g) \quad &\text{for } \lambdan + c (\ujumpn - g) < 0 \\
    \lambdan \quad &\text{otherwise}
    \end{cases},
    \end{align}
    \begin{align}
    C_\tau = \begin{cases}
        c b \ujumpt \quad &\text{for } b > \modulus{\ujumpt + c \lambdat} \\
        b (\lambdat + c \ujumpt) - \lambdat \modulus{\lambdat + c \ujumpt} \quad &\text{for } 0 < b \leq \modulus{\ujumpt + c \lambdat} \\
        \lambdat \quad &\text{for } b \leq 0 
    \end{cases}.
\end{align}
\end{subequations}
The fixed stress sequential iterative scheme addresses the linearized problem which arises within Newton's method iterations. With an abuse of notation we use the same letters to denote the finite increments of the original variables, which become the unknowns of the linearized problem:
$\disp \in L^\infty(0, T; V)$, 
$p \in L^\infty(0, T; L^2(\Omega))$, 
$p_c \in L^\infty(0, T; H^\frac{1}{2}(\Cfrac))$,
$\vec{z} \in L^2(0, T; \mathbf{Z})$,
$\vec{\zeta} \in L^2(0, T; \mathbf{Z}_\Cfrac)$
and $\lambdafull \in L^2(0, T; L^2(\Cfrac^D))$.
The fixed stress iterative scheme reads:

\noindent
\textbf{Step A}: Given $\disp^k$ and $\lambdafull^{k}$, solve for $p^{k+1}, \vec{z}^{k+1}, p_c^{k+1}, \vec{\zeta}^{k+1}$:
\begin{subequations}
\begin{multline}
\label{eq:mass_balance_fs}
\forall \theta \in L^2(\Omega), \quad
\left(
\partial_t \left(
(\dfrac{1}{M} + c_f \varphi_0 + \dfrac{\alpha^2}{\Lame} )  p^{k+1} 
+ \alpha \nabla \cdot  \disp^{k}
\right), \theta
\right) \\
+ \left(
\nabla \cdot  \vec{z}^{k+1}, \theta
\right)
= \left(
-\dfrac{\alpha}{\Lame} \partial_t  \sigma_v^k,
\theta
\right)
+ \left(
\text{rhs}, \theta
\right),
\end{multline}
\begin{equation}
\label{eq:flux_fs}
\forall \vec{q} \in \mathbf{Z}, \quad
\left(
\visc \mathbf{K}^{-1}  \vec{z}^{k+1}, \vec{q}
\right) 
= ( p^{k+1}, \nabla \cdot \vec{q})
- \left(
 p_c^{k+1}, \jump{\vec{q}}_C \cdot \vec{n}^+
\right)_C
+ \left(
\text{rhs}, \vec{q}
\right),
\end{equation}
\begin{multline}
\label{eq:mass_balance_frac_fs}
\forall \theta_C \in H^{\frac{1}{2}}(C), \quad
\left(
\partial_t (\left(c_{fc} + \gamma_c)  p_c^{k+1}
\right), \theta_c
\right)_C
+ \left(
\frac{1}{12} \overline{\nabla} \cdot  \vec{\zeta}^{k+1}, \theta_c
\right)_C \\
- \left(
\jump{ \vec{z}^{k+1}}_C \cdot \vec{n}^+, \theta_c
\right)_C
= \left(
\partial_t \sigma^k_{fv}, \theta_c
\right)_C
+ \left(
\text{rhs}, \theta_c
\right)_C ,
\end{multline}
\begin{equation}
\label{eq:flux_frac_fs}
\forall \vec{q_c} \in \mathbf{Z}_C, \quad
\left(
\visc \mathbf{K}^{-1}_c  \vec{\zeta}^{k+1}, \vec{q}_c
\right)_C
= \left(
 p_c^{k+1}, \overline{\nabla}\cdot \vec{q}_c
\right)_C
+ \left(
\text{rhs}, \vec{q}_c
\right)_C,
\end{equation}
\end{subequations}
with $\mathbf{K}_c$ denoting permeability of the fracture. Once the flow is computed, we update the displacements and contact mechanics.

\noindent
\textbf{Step B}: Given $ p^{k+1},  \vec{z}^{k+1},  p_c^{k+1},  \vec{\zeta}^{k+1}$, compute 
$ \disp^{k+1}$ and $ \lambdafull^{k+1}$:
\begin{subequations}
\begin{multline}
\label{eq:elasticity_fs}
\forall \vec{v} \in V, \quad 
2 G \left(  \varepsilon(\disp^{k+1}), \varepsilon(\vec{v}) \right) 
+ \Lame ( \nabla \cdot  \disp^{k+1}, \nabla \cdot \vec{v} ) \\
- \alpha (  p^{k+1}, \nabla \cdot \vec{v} )
+ (  p_c^{k+1}, \jump{\vec{v}}_\Cfrac \cdot \vec{n}^+ )_\Cfrac 
- (  \lambdafull^{k+1}, \vec{v} )_\Cfrac
= ( \text{rhs}, \vec{v} ),
\end{multline}
\begin{equation}
\label{eq:contact_normal_fs}
    \forall \theta_c \in L^2(C) \quad
    ( C_\text{n}^{k+1}, \theta_c)_\Cfrac  = (\text{rhs}, \theta_c)_\Cfrac;
    \quad
    ( C_\tau^{k+1}, \theta_c)_\Cfrac  = (\text{rhs}, \theta_c)_\Cfrac.
\end{equation}
\end{subequations}
where $\varepsilon(\disp)$ denotes the small deformation tensor.

We denote the fixed quantities from the previous Newton iterations with the subscript 0, e.g. $\ujumpt_0$. $C_\text{n}^{k+1}$ and $C_\tau^{k+1}$ are the finite increments of the complementarity functions \cref{eq:contact_equalities_for_fs}. To write them down, we first define these auxiliary coefficients:
\begin{equation}
\label{eq:auxiliary_coefficients}
    \vec{y} \coloneq {\lambdat}_0 + c \ujumpt_0;
    \quad
    \epsilon \coloneq \modulus{{\lambdat}_0} - F \modulus{{\lambdan}_0};
    \quad
    \beta_B \coloneq \dfrac{\Delta u_\text{max}^2 K_{n}}{(\Delta u_\text{max} K_{n} - {\lambdan}_0)^2},
\end{equation}
where $\epsilon$ is interpreted as Coulomb's friction law mismatch which may appear during Newton iterations, and $\beta_B$ corresponds to the linearized Barton-Bandis fracture elasticity constitutive law, see \cref{eq:gap_function}. 

The linearized complementarity functions have the following form:
\begin{subequations}
\label{eq:contact_linearized}
\begin{align}
    C_\text{n}^{k+1} = \begin{cases}
    - c (\ujumpn^{k+1} - \beta_{B} \lambdan^{k+1}) \quad &\text{for } \lambdan_0 + c (\ujumpn_0 - g_0) < 0 \\
    \lambdan^{k+1} \quad &\text{otherwise}
    \end{cases},
\end{align}
\begin{align}
    C_\tau^{k+1} = \begin{cases}
        c (-F \ujumpt_0 \lambdan^{k+1} + b_0 \ujumpt^{k+1}) \quad &\text{for } b_0 > \modulus{\ujumpt_0 + c \lambdat_0} \\
        c \epsilon \ujumpt^{k+1} 
        + (\modulus{\vec{y}} + \epsilon) \lambdat^{k+1}
        - F \vec{y} \lambdan^{k+1} 
        \quad 
        &\text{for } 0 < b_0 \leq \modulus{\ujumpt_0 + c \lambdat_0}
        \\
        \lambdat^{k+1} \quad &\text{for } b_0 \leq 0 
    \end{cases},
\end{align}
\end{subequations}
where $b_0 \coloneq -F \lambdan_0$ is a fixed friction bound. Note that the state of the complementarity functions is fixed on the previous linearization step according to Newton's method.
\subsection{Contraction proof}
We list all the assumptions we make to prove that the scheme is a contraction.
\begin{assumption}
\label{assumption:1}
For the mechanical modeling, the reservoir matrix is a homogeneous, isotropic, and saturated poroelastic medium. The \LameName coefficients $\Lame$ and $G$, the dimensionless Biot coefficient $\alpha$, the reference density for the fluid $\rhoRef$, and the viscosity $\visc$ are positive. 
\end{assumption}
\textit{Remark:}
The assumption of homogeneity is a simplification; extension to a heterogeneous medium is straightforward, as is the extension to a tensor $\alpha$.
\begin{assumption}
The fluids are assumed slightly compressible and for simplicity, the densities of the fluids are assumed to be linear functions of pressure.
\end{assumption}
\textit{Remark:}
 The linearized density terms are not crucial for the proof, but just a matter of convenience to match the notation of \cite{girault_convergence_2016}.
\begin{assumption}
The absolute permeability tensors, $\mathbf{K}$ and $\mathbf{K}_c$ are assumed to be symmetric, bounded, uniformly positive definite in space and constant in time.
\end{assumption}
\begin{assumption}
\label{assumption:sticking_ut_zero}
For the linearized sticking case $b_0 > \modulus{\ujumpt_0 + c \lambdat_0}$, the following holds:
\begin{equation}
    \ujumpt_0 = 0
\end{equation}
\end{assumption}
\textit{Remark:}
Due to the regularization inherent in the Augmented Lagrangian formulation of the contact conditions, the tangential jump across the fracture can be non-zero even if the fracture is in a sticking case. Such a nonphysical configuration is disallowed by this assumption.
\begin{assumption}
\label{assumption:coercivity}
The \LameName coefficients: bulk modulus $\Lame$, shear modulus $G$, the compressibilities of the fluid, $c_f$ and $c_{fc}$, the Biot constants $\alpha$ and $M$, and the initial porosity $\varphi_0$ satisfy:
\begin{equation}
\label{eq:convergence_condition}
    4 G \mathcal{D} \geq \dfrac{\Lame}{c_{fc} \alpha^2} 
    \left( 
    \dfrac{1}{M} + c_{f} \varphi_0
    \right)
    + 2 \Lame \left(
    \dfrac{1}{\beta^\text{\textnormal{min}}_B}
    + \dfrac{c \epsilon^\textnormal{max}}{(\modulus{\vec{y}} + \epsilon)^\textnormal{min}}
    + \dfrac{F \modulus{\vec{y}}^\textnormal{max}}{(\modulus{\vec{y}} + \epsilon)^\textnormal{min} \beta_B^\textnormal{min}}
    \right)
    ,
\end{equation}
where $\mathcal{D}$ is a product of optimal constants in Korn's and trace inequalities.
\end{assumption}
\textit{Remark:}
Such requirement arises in the convergence study of the scheme. The factors in \cref{eq:convergence_condition} represent physical parameters, including $\mathcal{D}$ that takes into account the geometrical setting of the problem. Its presence suggests that the shape and location of the fracture plays a role in the convergence of the iterative scheme; see \cite{girault_convergence_2016} for details. The second term in the right-hand side reflects the contact mechanics influence, where the ``min'' and ``max'' superscripts refer to the extrema values of the contact state-related coefficients on the known Newton iteration.
We define the difference between two iterates: $\delta_2 \xi = \xi^{k+1} - \xi^{k}; \delta_1 \xi = \xi^{k} - \xi^{k-1}$, where $\xi$ may stand for $p, p_c, \disp, \vec{z}, \vec{\zeta}$, $\lambdafull$, etc. 
We start the contraction proof by first addressing the mass balance and fluid flux equation.
Our first result is quite close to the result from \cite{girault_convergence_2016}, which accounts for all parts of the problem except for the contact mechanics term. We spare the details and mention only the key steps in the proof.

\begin{lemma}
    Under the assumptions \ref{assumption:1} - \ref{assumption:coercivity}, the following holds:
    \begin{multline}
    \label{eq:girault_contraction}
        \norm{\partial_t \delta_2 \sigma_v}^2_{\Omega_t} 
        + \Lame^2 \norm{\nabla \cdot \partial_t \delta_2 \disp}^2_{\Omega_t}
        + \dfrac{1}{\beta} \norm{\visc^\frac{1}{2} \mathbf{K}^{-\frac{1}{2}} \delta_2 \vec{z}}^2_{L^2(\Omega \backslash \Cfrac)} \\
        + \dfrac{1}{12 \beta} \norm{\visc^\frac{1}{2} \mathbf{K}_c^{-\frac{1}{2}} \delta_2 \vec{\zeta}}^2_{L^2(\Cfrac)}
        + \norm{\partial_t \delta_2 \sigma_f}^2_{\Cfrac^t}
        + \left(4 G \Lame \mathcal{D} - \dfrac{1}{\chi^2} \right) \norm{\partial_t \delta_2 \ujumpn}^2_{\Cfrac^t} \\
        + 4 G \Lame \mathcal{D} \norm{\partial_t \delta_2 \ujumpt}^2_{\Cfrac^t}
        \leq \dfrac{1}{\beta^2 \Lame^2} \norm{\partial_t \delta_1 \sigma_v}^2_{\Omega^t} \\
        + \dfrac{\chi^2}{\beta_c \beta} \norm{\partial_t \delta_1 \sigma_f}^2_{\Omega^t}
        + 2 \Lame \int^t_0 \left(
            \partial_t \delta_2 \lambdafull, \partial_t \delta_2 \ujump
        \right)_C d \tau.
    \end{multline}
    
\end{lemma}
\begin{proof}
The proof consists of two steps. We first prove the following estimate:
    \begin{equation}
    \label{eq:lemma_flow}
    \begin{aligned}
        \norm{\alpha \partial_t \delta_2 p}^2_{L^2(\Omega \backslash \Cfrac)}
        + \dfrac{1}{\beta} \dfrac{d}{dt} \norm{\mu^\frac{1}{2}_f \mathbf{K}^{-\frac{1}{2}} \delta_2 \vec{z}}^2_{L^2(\Omega \backslash \Cfrac)}
        + 2 \dfrac{\beta_c}{\beta} \norm{\partial_t \delta_2 p_c}^2_{L^2(\Cfrac)} \\
        + \dfrac{1}{12 \beta} \dfrac{d}{dt}
        \norm{\visc^\frac{1}{2} \mathbf{K}_c^{-\frac{1}{2}} \delta_2 \vec{\zeta}}
        \leq \dfrac{1}{\beta^2 \Lame^2}
        \norm{\partial_t \delta_1 \sigma_v}^2_{L^2(\Omega \backslash \Cfrac)} \\
        + \dfrac{2}{\beta} (\gamma_c \partial_t \delta_1 p_c - \partial_t \delta_1 \ujumpn, \partial_t \delta_2 p_c)_C.
    \end{aligned}
    \end{equation}

    We take the flow model \cref{eq:mass_balance_fs,eq:flux_fs,eq:mass_balance_frac_fs,eq:flux_frac_fs} in the matrix and fractures and exploit the linearity to get the difference of two successive iterates. We choose for the test functions $\partial_t \delta_2 p, \delta_2 \vec{z}, \partial_t \delta_2 p_c$ and $\delta_2 \vec{\zeta}$, respectively.
    An application of  Young's inequality and differentiating the matrix flux term in the matrix \eqref{eq:flux_fs} and in fractures \eqref{eq:flux_frac_fs} with respect to time give the desired result. The more detailed derivation is given in Section 4, Step 1 of \cite{girault_convergence_2016}.

Next, we address the elasticity equation. The procedure is the same as in Section 4, Step 2 of \cite{girault_convergence_2016}, but we add the contact mechanics term. The following estimate holds:
\begin{equation}
\label{eq:lemma_elasticity}
\begin{aligned}
    4 G \Lame \mathcal{D} \norm{\partial_t \delta_2 \ujump}^2_{\Cfrac^t} 
    + 2 \Lame^2 \norm{\nabla \cdot \partial_t \delta_2 \disp}^2_{\Omega_t} \\
    - 2 \Lame \int^t_0 \alpha \left(
    \partial_t \delta_2 p, \nabla \cdot \partial_t \delta_2 \disp
    \right) d \tau
    - 2 \Lame \int^t_0 \left(
        \partial_t \delta_2 p_c, \partial_t \delta_2 \ujumpn
    \right)_C d \tau \\
    \leq 2 \Lame \int^t_0 \left(
        \partial_t \delta_2 \lambdafull, \partial_t \delta_2 \ujump
    \right)_C d \tau,
\end{aligned}
\end{equation}
where the abbreviation $\Omega_t$ stands for $L^2((\Omega \backslash \Cfrac) \times (0, t)).$
Again, we take the difference of successive iterates of the elasticity Eq. \eqref{eq:elasticity_fs}, differentiate it with respect to time and choose for the test function $\partial_t \delta_2 \disp$. We combine Eqs. \eqref{eq:lemma_flow} and \eqref{eq:lemma_elasticity} integrated over $t$, rearrange the terms and infer:
    \begin{multline}
        \norm{\partial_t \delta_2 \sigma_v}^2_{\Omega_t} 
        + \Lame^2 \norm{\nabla \cdot \partial_t \delta_2 \disp}^2_{\Omega_t}
        + \dfrac{1}{\beta} \norm{\visc^\frac{1}{2} \mathbf{K}^{-\frac{1}{2}} \delta_2 \vec{z}}^2_{L^2(\Omega \backslash \Cfrac)} \\
        + \dfrac{1}{12 \beta} \norm{\visc^\frac{1}{2} \mathbf{K}_c^{-\frac{1}{2}} \delta_2 \vec{\zeta}}^2_{L^2(\Cfrac)}
        + \left\{
        \dfrac{2 \beta_c}{\beta} \norm{\partial_t \delta_2 p_c}^2_{\Cfrac^t} 
        \right. \\ \left.
        - 2 \Lame \int_0^t \left( \partial_t \delta_2 p_c, \partial_t \delta_2 \ujumpn \right)_\Cfrac d \tau
        + 4 G \Lame \mathcal{D} \norm{\partial_t \delta_2 \ujumpn}^2_{\Cfrac^t}
        \right\} \\
        + 4 G \Lame \mathcal{D} \norm{\partial_t \delta_2 \ujumpt}^2_{\Cfrac^t}
        \leq \dfrac{1}{\beta^2 \Lame^2} \norm{\partial_t \delta_1 \sigma_v}^2_{\Omega^t} 
        \\
        + \int_0^t \dfrac{2}{\beta} \left( \gamma_c \partial_t \delta_1 p_c - \partial_t \delta_1 \ujumpn, 
        \partial_t \delta_2 p_c \right)
        d \tau 
        + 2 \Lame \int^t_0 \left(
            \partial_t \delta_2 \lambdafull, \partial_t \delta_2 \ujump
        \right)_C d \tau.
    \end{multline}
    Let us define the term in the fractures, analogous to the volumetric stress in the matrix: $\chi \partial_t \delta_1 \sigma_f = \gamma_c \partial_t \delta_1 p_c - \partial_t \delta_1 \ujumpn$. After rearranging the terms and completing the square with the new term, we complete the proof. For a more detailed procedure, the reader is referred to Lemma 4.1 of \cite{girault_convergence_2016}.
\end{proof}
Next, we address the contact mechanics term, which is located on the right-hand side of Eq. \eqref{eq:lemma_elasticity}.
\begin{lemma}
    Given the  assumptions \ref{assumption:1} - \ref{assumption:coercivity}, the following holds:
    \begin{multline}
        \label{eq:contact_mechanics_bound}
        \left( \partial_t \delta_2 \lambdafull, \partial_t \delta_2 \ujump \right)_\Cfrac 
        \leq \\
        \left( 
            \dfrac{1}{\beta^\text{\textnormal{min}}_B} + \dfrac{c \epsilon^\text{\textnormal{max}}}{(\modulus{\vec{y}} + \epsilon)^\text{\textnormal{min}}}
            + \dfrac{F \modulus{\vec{y}}^\text{\textnormal{max}}}{(\modulus{\vec{y}} + \epsilon)^\text{\textnormal{min}} \beta_B^\text{\textnormal{min}} }
        \right)
        \norm{\partial_t \delta_2 \ujump}^2_{L^2(\Cfrac)},
    \end{multline}
    where the coefficients $\vec{y}, \epsilon$ and $\beta_B$ are given by \cref{eq:auxiliary_coefficients}.
\end{lemma}
\begin{proof}
    We start with considering the scalar product:
    \begin{equation}
    \label{eq:contact_normal_tangential}
    \left( \partial_t \delta_2 \lambdafull, \partial_t \delta_2 \ujump \right)_\Cfrac = \left( \partial_t \delta_2 \lambdan, \partial_t \delta_2 \ujumpn \right)_\Cfrac + \left( \partial_t \delta_2 \lambdat, \partial_t \delta_2 \ujumpt \right)_\Cfrac.
    \end{equation}
    We take the difference of two successive iterates of Eq. \eqref{eq:contact_linearized}, differentiate w.r.t time and substitute the test function $\partial_t \delta_2 \ujump$.
    
    For the normal component, we define the closed and open subsets of the contact boundary: $\Cfrac_\text{n,cl} \subset \Cfrac: \{ \lambdan + c (\ujumpn - g) < 0 \}$ and $\Cfrac_\text{n,op} = \Cfrac \backslash \Cfrac_\text{n,cl}$, and obtain:
    \begin{equation}
    \begin{aligned}
        \begin{cases} 
        \left(        
        \beta_B \partial_t \delta_2 \lambda_n,
        \partial_t \delta_2 \ujumpn
        \right)_\mathcal{C_\text{n,cl}}
        &= \norm{
        \partial_t \delta_2 \ujumpn
        }^2_{L^2(\Cfrac_\text{n,cl})},
         \\
         \left(
         \partial_t \delta_2 \lambdan,
         \partial_t \delta_2 \ujumpn
         \right)_\mathcal{C_\text{n,op}} 
         &= 0.
        \end{cases}\\
    \end{aligned}
    \end{equation}
    Since the test function is in $L^2$, the appropriate restrictions can be chosen to make each of the equalities true, not only their sum. The strictly positive coefficient $\beta_B$ is bound by its minimum:
    \begin{equation}
        \left(        
        \partial_t \delta_2 \lambda_n,
        \partial_t \delta_2 \ujumpn
        \right)_\mathcal{C_\text{n,cl}}
        \leq \dfrac{1}{\beta_B^\text{min}}\norm{
        \partial_t \delta_2 \ujumpn
        }^2_{L^2(\Cfrac_\text{n,cl})}.
    \end{equation}
    We take the sum of the components and obtain the following for the normal component:
    \begin{equation}
        \label{eq:sticking_normal_bound}
        \left(        
        \partial_t \delta_2 \lambda_n,
        \partial_t \delta_2 \ujumpn
        \right)_\Cfrac
        \leq \dfrac{1}{\beta_B^\text{min}}\norm{
        \partial_t \delta_2 \ujumpn
        }^2_{L^2(\Cfrac_\text{n,cl})}
        \leq \dfrac{1}{\beta_B^\text{min}}\norm{
        \partial_t \delta_2 \ujumpn
        }^2_{L^2(\Cfrac)}.
    \end{equation}
    Next, we address the tangential component of Eq. \eqref{eq:contact_linearized}. Again, we take the difference of two successive iterations, differentiate w.r.t. time, and split the contact boundary into three subsets corresponding to the sticking, sliding and open states: $\Cfrac_\text{st} \subset \Cfrac : \{ b > \modulus{\vec{y}}\}$, $\Cfrac_\text{sl} \subset \Cfrac : \{ 0 < b \leq \modulus{\vec{y}}\}$ and $\Cfrac_{\tau,\text{op}} = \Cfrac \backslash (\Cfrac_\text{st} \cup \Cfrac_\text{sl})$. For the open state:
    \begin{equation}
        \left( \partial_t \delta_2 \lambdat, \partial_t \delta_2 \ujumpt \right)_{\Cfrac_{\tau,\text{op}}} = 0
    \end{equation}
    For the sliding case, we bound the expression:
    \begin{equation}
        \label{eq:contact_bound_sliding}
        (\vec{\modulus{y} + \epsilon})^\text{min}
        \left(
        \partial_t \delta_2 \lambdat, \partial_t \delta_2 \ujumpt
        \right)_{\Cfrac_\text{sl}}
        \leq c \epsilon^\text{max}  \norm{\partial_t \delta_2 \ujumpt}^2_{L^2(\Cfrac_\text{sl})}
        + F \modulus{\vec{y}} \left(
        \modulus{\partial_t \delta_2 \lambdan}, \modulus{\partial_t \delta_2 \ujumpt}
        \right)_{\Cfrac_\text{sl}},
    \end{equation}
    where we used the following relation to bound the last term: $\vec{y} \cdot \delta_2 \ujumpt \leq \modulus{\vec{y}} \modulus{\delta_2 \ujumpt}$. Note that for the sliding case, $\modulus{\vec{y}} > 0$ according to $\Cfrac_\text{sl}$. 
    Next, we consider the term $\left(\partial_t \delta_2 \ujumpn, \modulus{\partial_t \delta_2 \lambdat}\right)_{\Cfrac_\text{sl}}$, which can be bounded by the same procedure as in \cref{eq:sticking_normal_bound} by taking the test function $\partial_t \delta_2 \ujump$ on $\Cfrac_\text{sl}$ and 0 on $\Cfrac \backslash \Cfrac_\text{sl}$:
    \begin{equation}
        \left(        
        \modulus{\partial_t \delta_2 \lambda_n},
        \modulus{\partial_t \delta_2 \ujump}
        \right)_{\Cfrac_\text{sl}}
        \leq \dfrac{1}{\beta_B^\text{min}}\norm{
        \partial_t \delta_2 \ujump
        }^2_{L^2(\Cfrac_\text{sl})}.
    \end{equation}
    Substitution into \cref{eq:contact_bound_sliding} gives the following:
    \begin{equation}
       \left(
        \partial_t \delta_2 \lambdat, \partial_t \delta_w \ujumpt
        \right)_{\Cfrac_\text{sl}}
        \leq \frac{c \epsilon^\text{max} }{\beta_\text{min}} \norm{\partial_t \delta_2 \ujumpt}^2_{L^2(\Cfrac_\text{sl})}
        + \dfrac{F \modulus{\vec{y}}}{(\vec{\modulus{y} + \epsilon})^\text{min}\beta_B^\text{min}} \norm{
        \partial_t \delta_2 \ujump
        }^2_{L^2(\Cfrac_\text{sl})}.
    \end{equation}
    The next step is to consider the sticking case and substitute the test function $\partial_t \delta_2 \lambdat \in L^2(\Cfrac^D)$. We bound the expression with $\modulus{\lambdan}^\text{min}$ and $\modulus{\ujumpt}^\text{max}$:
    \begin{equation}
        \modulus{\lambdan}^\text{min} \left(
         \partial_t \delta_2 \ujumpt, \partial_t \delta_2 \lambdat
        \right)_{\Cfrac_\text{st}}
        \leq \modulus{\ujumpt}^\text{max} \left(
        \partial_t \delta_2 \lambdan, \partial_t \delta_2 \lambdat
        \right)_{\Cfrac_\text{st}}
        \leq \modulus{\ujumpt}^\text{max} \norm{\partial_t \delta_2 \lambdafull}_{L^2(\Cfrac_\text{st})}.
    \end{equation}
    We recall \cref{assumption:sticking_ut_zero} which dictates that $\ujumpt = 0$ for the sticking case, thus
    \begin{equation}
        \left(
         \partial_t \delta_2 \ujumpt, \partial_t \delta_2 \lambdat
        \right)_{\Cfrac_\text{st}}
        \leq 0.
    \end{equation}
    Finally, we substitute the normal and the tangential inequalities into Eq. \eqref{eq:contact_normal_tangential}:
    \begin{multline}
        \left( \partial_t \delta_2 \lambdafull, \partial_t \delta_2 \ujump \right)_\Cfrac 
        \leq \dfrac{1}{\beta_B^\text{\textnormal{min}}} \norm{\partial_t \delta_2 \ujumpn}^2_{L^2(\Cfrac)}
        + \dfrac{c \epsilon^\text{\textnormal{max}}}{(\modulus{\vec{y}} + \epsilon)^\text{\textnormal{min}}}
        \norm{\partial_t \delta_2 \ujumpt}^2_{L^2(\Cfrac)} \\
        + \dfrac{F \modulus{\vec{y}}^\text{\textnormal{max}}}{(\modulus{\vec{y}} + \epsilon)^\text{\textnormal{min}} \beta_B^\text{\textnormal{min}} }
        \norm{\partial_t \delta_2 \ujump}^2_{L^2(\Cfrac)}.
    \end{multline}
    We combine the normal and the tangential terms and complete the proof.
\end{proof}
We now have all the ingredients to demonstrate the contraction of the fixed stress splitting scheme.
\begin{theorem}
    \label{theorem:convergence}
    Under the assumptions \ref{assumption:1} - \ref{assumption:coercivity}, the fixed stress splitting scheme defined by \cref{eq:mass_balance_fs,eq:mass_balance_frac_fs,eq:flux_fs,eq:flux_frac_fs,eq:elasticity_fs,eq:contact_normal_fs} is a contraction.
\end{theorem}
\begin{proof}
    We substitute \cref{eq:contact_mechanics_bound} into \cref{eq:girault_contraction} and move the contact mechanics term to the left-hand side:
    \begin{multline}
        \norm{\partial_t \delta_2 \sigma_v}^2_{\Omega_t} 
        + \Lame^2 \norm{\nabla \cdot \partial_t \delta_2 \disp}^2_{\Omega_t}
        + \dfrac{1}{\beta} \norm{\visc^\frac{1}{2} \mathbf{K}^{-\frac{1}{2}} \delta_2 \vec{z}}^2_{L^2(\Omega \backslash \Cfrac)} \\
        + \dfrac{1}{12 \beta} \norm{\visc^\frac{1}{2} \mathbf{K}_c^{-\frac{1}{2}} \delta_2 \vec{\zeta}}^2_{L^2(\Cfrac)}
        + \norm{\partial_t \delta_2 \sigma_f}^2_{\Cfrac^t}
        + \left( 
            4 G \mathcal{D} \Lame - \dfrac{1}{\chi^2} \right. \\
            \left. - 2 \Lame (
                \dfrac{1}{\beta^\text{\textnormal{min}}_B}
                + \dfrac{c \epsilon^\textnormal{max}}{(\modulus{\vec{y}} + \epsilon)^\textnormal{min}}
                + \dfrac{F \modulus{\vec{y}}^\textnormal{max}}{(\modulus{\vec{y}} + \epsilon)^\textnormal{min} \beta_B^\textnormal{min}}
                )
        \right)
        \norm{\partial_t \delta_2 \ujump}^2_{\Cfrac^t} \\
        \leq \dfrac{1}{\beta^2 \Lame^2} \norm{\partial_t \delta_1 \sigma_v}^2_{\Omega^t}
        + \dfrac{\chi^2}{\beta_c \beta} \norm{\partial_t \delta_1 \sigma_f}^2_{\Omega^t}
    \end{multline}
    As shown in \cite{girault_convergence_2016}, $\Lame \beta > 1$ and $\chi^2 / \beta_c \beta < 1$, so \cref{theorem:convergence} implies that for any $t$ in $(0, t)$ $\partial_t \sigma_v$ and $\partial_t \sigma_f$ are Cauchy sequences in $L^2((\Omega \backslash \Cfrac) \times (0, t))$ and $L^2((\Cfrac \times (0, t))$, respectively, with geometric convergence. It also implies that $\partial_t \nabla \cdot \disp$ and $\partial_t \ujump$ converge geometrically in $L^2((\Omega \backslash \Cfrac) \times (0, t))$ and $L^2((\Cfrac \times (0, t))$, respectively. Thus, we can obtain strong convergence of all the considered sequences, see \cite{girault_convergence_2016} for the details. The existence of the limiting functions in the relevant spaces comes from the completeness of these spaces.
\end{proof}
\Cref{theorem:convergence} shows that using the coefficient $\gamma_c$, which is defined in \cref{eq:aux_coefs_gamma_c}, to stabilize the fluid mass balance equation in fractures together with the classical fixed stress stabilization for the fluid mass balance in the ambient dimension provides a conditionally stable iterative scheme.
A tedious calculation shows that $\gamma_c$ is equivalent to $l_\text{frac}$ in \cref{eq:fixed_stress_coefs}.
This indicates that $l_\text{frac}$ is a proper approximation to build the second-level Schur complement and justifies its usage. 
We also experimentally observe that the value of the fixed stress stabilization coefficient in the ambient dimension can be adjusted to provide faster convergence according to \cite{BOTH2017101,mikelic_convergence_2013}, so we use the sharper coefficient of $l_\text{mat}$ in \cref{eq:fixed_stress_coefs}.
\begin{corollary}
\label{corollary:1}
    \Cref{assumption:coercivity} can be simplified if we consider the physically possible states, where $\epsilon = 0$:
    \begin{equation}
    2 G \mathcal{D} \geq
    \dfrac{1}{2 c_{fc} \alpha^2} 
    \left( 
    \dfrac{1}{M} + c_{f} \varphi_0
    \right)
    + \dfrac{1 + F \modulus{\vec{y}}^\textnormal{max}/\modulus{\vec{y}}^\textnormal{min}}{\beta^\text{\textnormal{min}}_B}
    \end{equation}
\end{corollary}
\Cref{theorem:convergence} and \cref{corollary:1} indicate that convergence depends on the friction coefficient $F$, the contact traction $\lambdafull$, the tangential displacement velocity $\ujumpt$ and the Barton-Bandis contact elasticity coefficients. The effect of the mismatch in Coulomb's friction law $\epsilon$, which is caused by the Augmented Lagrangian formulation, influences the contraction by the magnitude of the numerical method constant $c$. 

\subsection{Transferring the convergence result to the preconditioner}
\label{sec:connection_preconditioner_sequential_scheme}

We have obtained the contraction for the sequential iterative scheme of the simplified linearized model. This scheme can be reinterpreted as a preconditioned Richardson iteration for the fully coupled problem \cite{white_block-partitioned_2016}, although with constant fracture permeability. 
To compare the preconditioner that corresponds to the sequential scheme with the one presented in \Cref{sec:preconditioner}
 we present them as linear block matrices:
\begin{equation}
\label{eq:p1_p2}
\hat{P} = \left[\begin{array}{ccc;{4pt/0pt}cc}
    J_{11} & J_{12} &       &       &           \\
    J_{12} & J_{22} & J_{23}&       & J_{25}    \\
           & J_{32} & J_{33}&       & J_{35}    \\ \hdashline[4pt/0pt]
           &        &       & J_{44}& J_{45}    \\
           &        &       & J_{54}& S_{55}
\end{array}\right];
\quad
P = \left[\begin{array}{c;{4pt/8pt}cc;{4pt/0pt}c;{4pt/8pt}c}
  \tilde{J}_{11} & J_{12}  &       &       &           \\ \hdashline[4pt/8pt]
                 & S^1_{22}& J_{23}&       & J_{25}    \\
                 & S^1_{32}& J_{33}&       & J_{35}    \\ \hdashline[4pt/0pt]
                 &         &       & J_{44}& J_{45}    \\ \hdashline[4pt/8pt]
                 &         &       &       & S^{3}_{55}
\end{array}\right].
\end{equation}
The inverse of $\hat{P}$ denotes the iteration of the sequential iterative scheme, and the inverse of $P$ denotes the application of the preconditioner in \Cref{sec:preconditioner}.
The solid lines indicate the fixed stress splitting, while the dashed lines in $P$ represent the other Schur complement-based decouplings of contact mechanics and interface fluid flow according to \Cref{sec:eliminating_contact_mechanics,sec:eliminating_intf_flow}. The submatrix $S_{55}$ corresponds to the fixed stress-stabilized fluid mass balance submatrix, the expressions for $S^3_{55}$, $S^1_{22}$ and $S^1_{32}$ are given in \Cref{sec:preconditioner}. Note that the second level Schur complement $S^2$ is not an explicit part of $P$, since it is used only to form $S^3$ and there is no coupling of the contact mechanics and the elasticity with the fluid interface flow.
The main difference between two algorithms is the following: $P$ is applied to the linearly transformed problem: $J Q_r$ according to \Cref{sec:linear_transformation}, thus we have invertible $\tilde{J}_{11}$ instead of singular $J_{11}$. The following shows that the fixed stress approximation is applicably unmodified to the linearly transformed problem:
Convergence of Richardson iteration implies that $\norm{I - J \hat{P}^{-1}} < 1$. The same holds for the linearly transform problem $\norm{I - J Q_r (\hat{P} Q_r)^{-1}} < 1$. For brevity, we recast $J$, $Q_r$ and $\hat{P}$ as 2 by 2 block matrices where the first row and column correspond to indices 1, 2, and 3, and the second row and column correspond to indices 4 and 5:
\begin{equation}
    J = \begin{bmatrix}
        A & B \\
        C & D
    \end{bmatrix};
    \quad
    Q_r = \begin{bmatrix}
        W &  \\
          & I
    \end{bmatrix};
    \quad
    \hat{P} = \begin{bmatrix}
        A &   \\
          & S_\text{fs}
    \end{bmatrix}
    \cdot
    \begin{bmatrix}
        I & A^{-1} B \\
          & I
    \end{bmatrix},
\end{equation}
where $S_\text{fs}$ is the fixed stress-based approximation of the corresponding Schur complement, and $W$ is an invertible matrix that contains the linear transformation according to \cref{eq:Qright}. We construct the transformed preconditioner $\hat{P} Q_r$ as the upper-triangular factorization of the transformed Jacobian:
\begin{equation}
    J Q_r = \begin{bmatrix}
        A W & B \\
        C W & D
    \end{bmatrix};
    \quad
    \hat{P} Q_r = \begin{bmatrix}
        A W &   \\
            & S_q
    \end{bmatrix}
    \cdot
    \begin{bmatrix}
        I & (AW)^{-1}B \\
          & I
    \end{bmatrix},
\end{equation}
where $S_q$ is the Schur complement of the transformed matrix: $S_q \coloneq D - C W (A W)^{-1} B$. Since both $W$ and $A$ are invertible, $S_q$ exactly matches the Schur complement of the untransformed problem: $D - C A^{-1} B$. Therefore, we can utilize the fixed stress-based approximation $S_q \approx S_\text{fs}$ unmodified after the linear transformation, and the preconditioned matrix produce the same eigenvalues spectrum before and after the linear transformation.
As the last step to connect
$\hat{P} Q_r$ and $P$, we replace the submatrices $AW$ and $S_\text{fs}$ in $\hat{P} Q_r$ with their upper triangular counterparts according to the block $\mathcal{LDU}$ decomposition, namely, we decouple the contact mechanics from the elasticity and the interface fluid flow from the fluid mass balance. This substitution provides a spectrally equivalent preconditioner if decoupling is done exactly, and the resulting matrix is $P$.
Therefore, the contraction of the fixed stress iterative scheme implies good performance of $P$ in exact arithmetic. 

\section{Numerical experiments}
\label{sec:numerical_experiments}
The code for the numerical experiments is available on GitHub\footnote{\href{https://github.com/Yuriyzabegaev/FTHM-Solver/}{https://github.com/Yuriyzabegaev/FTHM-Solver/}} as well as the Docker image \cite{zabegaev_2025_14609885}. It is written in Python and is based on the mixed-dimensional porous media simulation framework PorePy \cite{porepy2024}. The computational simplex grids are generated by Gmsh \cite{gmsh}, while
PETSc \cite{petsc4py,petsc_user_manual} is used as a core for the linear solver routines, providing GMRES and Richardson and ILU algorithms. The classical algebraic multigrid method is provided by HYPRE BoomerAMG \cite{hypre}. One V-cycle of AMG is used for the fluid mass and the mechanics subproblems. The most expensive part of the preconditioner is construction and application of the AMG for mechanics, having the strong threshold of $0.7$. The fluid mass subproblem AMG has the truncation factor set to $0.3$, all the other BoomerAMG configuration options remain default.
For better conditioning, the mass unit is taken to be $10^{10}$ kg, and the contact force primary variable $\lambdafull$ unit is scaled by Young modulus. 
In the experiments, we use the preconditioner either with Richardson iterations or with GMRES. The former applies the preconditioner at the left side, which corresponds to the sequential scheme algorithm according to \Cref{sec:connection_preconditioner_sequential_scheme}, while the latter uses the right preconditioner, as this allows us to control the unpreconditioned residual. 
The relative convergence criterion is set to $10^{-10}$ for both algorithms, while the absolute criterion is set to $10^{-10}$ for Richardson iterations and $10^{-15}$ for GMRES.
The difference is motivated by the different scaling of the preconditioned and unpreconditioned residuals. For the Newton iterations, we apply a relative residual criterion of $10^{-7}$. GMRES restarts after 30 iterations with the iteration limit of 90.
The numerical experiments are ran with three models: The first model considers a 2D medium with a single embedded fracture, the second model is also 2D but with multiple line fractures and intersections, while the final model is in a 3D domain with multiple plane fractures and their intersections. The domain is a square with 1 km sides in 2D, or a cube with 1 km sides in 3D. For all the models, the south side of the domain is fixed, while the other sides are free to move and are subject to constant compressive forces. The model has pressure boundary conditions set on all sides. The boundary pressure matches the reference pressure everywhere except for the east boundary, where a higher pressure is applied to create a pressure gradient that drives fluid flow through the model. No source terms are included in the model. The setup is designed to trigger the various states of the frictional contact problem (stick, slide and open) and allows us to investigate the robustness and efficiency of the preconditioner. The material parameters used in the experiments are given in \cref{tab:material_params}. The reference pressure is set at $1$ MPa.

\begin{figure}[htb]
    \centering
    \includegraphics[width=0.7\linewidth]{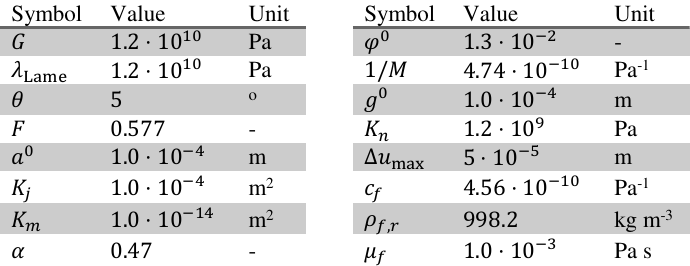}
    \caption{Material parameters used in the numerical experiments. $K_m$ denotes matrix isotropic permeability. The other symbols are defined in \Cref{sec:model_description}.}
    \label{tab:material_params}
\end{figure}
\subsection{Testing the iterative scheme, robustness against model parameters}
The first model contains a single horizontal fracture located halfway up the y-axis, extending from 0.2 km to 0.8 km along the x-axis. The east boundary pressure is set to 10 times the reference pressure. Compressive forces of $5$ MPa are applied normally to all boundaries. The north boundary has an additional tangential force component of $0.5$ MPa.
The simulated time is 3 days with 0.5-day time steps.
We first investigate the robustness of the sequential iterative scheme developed in \Cref{sec:convergence_study}.
To that end, we consider the restricted physics model presented therein and solve the linear system by Richardson iterations with the preconditioner $\hat{P}$ from \cref{eq:p1_p2} with exact solvers for the two submatrices. This linear solver algorithm exactly matches the sequential iterative scheme defined by \cref{eq:mass_balance_fs,eq:mass_balance_frac_fs,eq:flux_fs,eq:flux_frac_fs,eq:elasticity_fs,eq:contact_normal_fs}.
Though this algorithm is not scalable, the experiment aids in the exploration of the robustness of the practical preconditioner $P$. 
\Cref{theorem:convergence} indicates that the convergence of the sequential scheme
is affected by the friction coefficient $F$ and $\beta_B$ from \cref{eq:auxiliary_coefficients},
where the latter can be controlled by the elastic stiffness parameter $K_n$. To investigate the sensitivity of the convergence behavior, we run the simulation with $K_n \in \{0, 1.2\text{e}5,  1.2\text{e}9, 1.2\text{e}13, 1.2\text{e}17, 1.2\text{e}20\}$ Pa; the case $K_n=0$ is not covered by the analysis.
We also apply three different values of $F: \{0.1, 0.577, 0.8\}$.
The average number of Richardson iterations per Newton iteration are listed in \cref{tab:bb_friction}.
As can be seen, the iteration count primarily depends on $F$, with the higher value requiring more iterations to converge, aligning with the analysis. The variation of $K_n$ shows no clear effect on the convergence.
We repeated the experiment using the linearly transformed Jacobian from \Cref{sec:linear_transformation} and operator $P$ (with three Schur complements and direct subsolvers) instead of $\hat{P}$. The iteration counts were similar, so the results are omitted. We also observed that fracture position significantly impacts iteration counts: Positioning the fracture near the domain boundary require more iterations. This aligns with \cref{theorem:convergence}, where the effect is embedded in the $D$ coefficient. We also examined the magnitudes of the \cref{assumption:sticking_ut_zero} violation and the Coulomb's law mismatch $\epsilon$ from \cref{eq:auxiliary_coefficients} and did not find any correlation with the number of iterations.
\begin{table}[t]
    \centering
        \begin{tabular}{lrrrrrr}
        \toprule
         $K_n$ & $0$ & $1.2\text{e}5$ & $1.2\text{e}9$ & $1.2\text{e}13$ & $1.2\text{e}17$ & $1.2\text{e}20$ \\
        \midrule
        $F=0.1$ & 19.9 & 19.9 & 17.2 & 19.9 & 19.9 & 19.9 \\
        $F=0.577$ & 25.5 & 25.9 & 26.0 & 25.2 & 25.8 & 25.8 \\
        $F=0.8$ & 27.2 & 27.2 & 26.6 & 27.1 & 27.1 & 27.1 \\
        \bottomrule
        \end{tabular}
    \caption{Average count of Richardson iterations in simulations with various $K_n$ and $F$ coefficients.}
    \label{tab:bb_friction}
\end{table}
\subsection{Grid refinement of a 2D model}
Next, we apply the full physical model from \Cref{sec:model_description} to the same geometry and fix the values of the coefficients to $F = 0.577$ and $K_n = 1.2\text{e}9$ Pa. 
To ensure that some fracture cells are open we set the high pressure on the east boundary to 13 times the reference pressure. 
We investigate two algorithms: The first applies GMRES to the linearly transformed system from \Cref{sec:linear_transformation} using the $P$ preconditioner with direct subsolvers. The second replaces the direct subsolvers with AMG and ILU to test the practical algorithm.
The average number of GMRES iterations per Newton iteration as well as the Newton solver per time step are given in \cref{tab:model1_grid_refinement}. 
The number of GMRES iterations with direct subsolvers remains constant across the three available approximation levels, but the number of possible grid refinement levels is constrained by the direct subsolver's computational cost.
When replacing the direct subsolvers with AMG, the number of iterations grows slightly, with an increase of a factor of 2 between the coarsest and finest grid, corresponding to a three orders of magnitude increase in the number of degrees of freedom.

\begin{table}[t]
    \centering
        \begin{tabular}{lllllll}
        \toprule
        Total DoFs & 798 & 3018 & 17790 & 435492 & 757962 & 1111782 \\
        \midrule
        GMRES dir. & 3.4 & 3.3 & 3.2 & - & - & - \\
        GMRES AMG & 15.3 & 16.9 & 18.5 & 24.9 & 36.6 & 28.0 \\
        Newton iters. & 3.5 & 3.3 & 4.3 & 8.0 & 8.0 & 7.2 \\
        \bottomrule
        \end{tabular}
    \caption{2D geometry with 1 fracture. The average count of GMRES iterations per Newton iteration with the direct (``\textnormal{dir.}'') and AMG subsolvers. The last row is the average number of Newton iterations per time step.}
    \label{tab:model1_grid_refinement}
\end{table}
\subsection{Increasing geometric complexity}
\label{sec:experiment_2d_multiple_frac}
\begin{figure}[htb]
    \centering
    \includegraphics[width=1\linewidth]{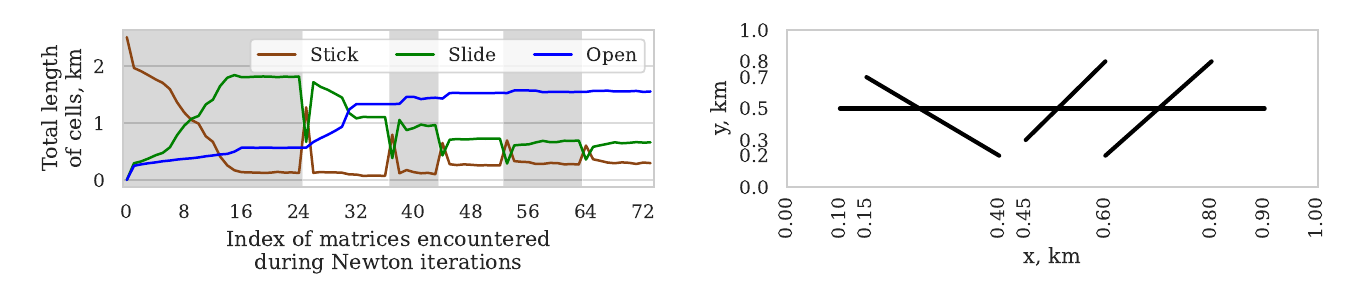}
    \caption{State (sticking, sliding or open) of fracture cells in the 2D experiment with multiple fractures (left). Shaded and white background colors distinguish different time steps. Geometry of the fractures in the same experiment (right).}
    \label{fig:fractures}
\end{figure}
Here, we investigate the performance of the preconditioner for the 2D model with multiple fractures. The geometry consists of one horizontal fracture extending from 0.1 km to 0.9 km and three differently rotated fractures, which intersect the horizontal one as shown in \cref{fig:fractures}.
The domain is subjected to normal compression of 3 MPa from the north, east, and west sides. The east side also has a high-pressure boundary condition, exceeding the reference pressure by a factor of 10. The full physical model is used to simulate the time of 1.5 hours with a time step of 0.25 hours. The dynamics of the fracture cell states is shown in \cref{fig:fractures}. Similarly to the previous experiment, we apply the preconditioner with the direct and the AMG subsolvers and refine the grid.
Results are listed in \cref{tab:model2_grid_refinement}. No increase in GMRES iteration count is observed for the direct subsolvers case, and only a small increase occurs for the AMG subsolvers. During this experiment, a few linear solver breakdowns occurred due to failure of AMG for mechanics. We attribute these failures to challenges in tuning AMG parameters for the MPSA discretization used in mixed-dimensional elasticity, as a perfectly scalable AMG method for this discretization has yet to be developed. However, the overall preconditioner does not depend on this specific discretization and can be applied with other spatial discretizations where elasticity solvers are more robust, such as finite element methods. In cases of solver breakdown, we accepted the linear solution and continued with the Newton iterations.
\begin{table}[t]
    \centering
        \begin{tabular}{lllllll}
        \toprule
        Total DoFs & 1260 & 3714 & 19512 & 445608 & 769842 & 1126848 \\
        \midrule
        GMRES dir. & 5.8 & 4.9 & 4.4 & - & - & - \\
        GMRES AMG & 34.1 & 33.4 & 35.4 & 39.7 & 38.8 & 40.4 \\
        Newton iters. & 5.8 & 6.3 & 7.5 & 10.3 & 10.3 & 12.3 \\
        \bottomrule
        \end{tabular}
    \caption{2D geometry with multiple fractures. The average count of GMRES iterations per Newton iteration with the direct (``\textnormal{dir.}'') and AMG subsolvers. The last row is the average number of Newton iterations per time step.
    }
    \label{tab:model2_grid_refinement}
\end{table}

\subsection{3D scalability}
In the final experiment, we perform a 3D simulation with the fractures extended from the previous experiment into the third dimension, starting at 0.2 km and ending at 0.8 km. The full physical model is used, except for the fracture permeability, which is kept constant to avoid Newton convergence issues. The fracture permeability is set equal to that of the surrounding domain. The east boundary is subjected to high pressure, 20 times greater than the reference pressure. 
All boundaries except the south are compressed by a normal force of 3 MPa, with the north boundary also subjected to a tangential force of 0.3 MPa. The simulated time is 1.5 weeks with a time step of 0.25 weeks. Again, we refine the grid to test the scalability of the preconditioner with the direct subsolvers and AMG. The results are given in \cref{tab:model3_grid_refinement}, where we observe almost no changes in GMRES iterations with grid refinement up to almost $10^6$ degrees of freedom. Notably, we observed no linear solver breakdowns this time.
\begin{table}[htb]
    \centering
        \begin{tabular}{lllllll}
        \toprule
        Total DoFs & 3154 & 6461 & 29002 & 174668 & 540730 & 938422 \\
        \midrule
        GMRES AMG & 31.0 & 29.5 & 26.7 & 28.6 & 27.7 & 28.8 \\
        Newton iters. & 6.0 & 8.0 & 9.5 & 13.8 & 15.0 & 15.0 \\
        \bottomrule
        \end{tabular}
    \caption{3D geometry with multiple fractures. The average count of GMRES iterations per Newton iteration with the AMG subsolvers. The last row is the average number of Newton iterations per time step.}
    \label{tab:model3_grid_refinement}
\end{table}
\section{Conclusion}
\label{sec:results}

In this study, we developed a robust and scalable preconditioner for solving the frictional contact poromechanics problem involving fractures. The proposed approach efficiently addresses the complex coupling between contact mechanics and poromechanics, particularly focusing on the saddle-point structure introduced by the contact mechanics equations. By combining the fixed stress method with a linear transformation approach for handling the saddle point system, we ensured the convergence and stability of the preconditioner. Our approach is applicable to any discretization of the contact (poro)mechanics problem which adopts Lagrange multiplier-based approaches.
Through numerical experiments, we demonstrated the preconditioner's scalability and robustness. The results showed that the preconditioner maintains computational efficiency across a range of problem sizes, from small-scale systems to large-scale systems of a size relevant in real-world applications.

\section*{Acknowledgment}
This project has received funding from the VISTA program, The Norwegian Academy of Science and Letters and Equinor and from
the Norwegian Research Council, Grant 308733.

\bibliography{references}

\begin{thebibliography}{10}

\bibitem{FAN20191054}
C.~Fan, D.~Elsworth, S.~Li, L.~Zhou, Z.~Yang, and Y.~Song.
\newblock Thermo-hydro-mechanical-chemical couplings controlling ch4 production and co2 sequestration in enhanced coalbed methane recovery.
\newblock {\em Energy J.}, 173:1054--1077, 2019.

\bibitem{LIU20191}
P.~Liu, T.~Zhang, and S.~Sun.
\newblock A tutorial review of reactive transport modeling and risk assessment for geologic co2 sequestration.
\newblock {\em Comput. Geosci.}, 127:1--11, 2019.

\bibitem{flemisch2024fluidflower}
B.~Flemisch, J.M. Nordbotten, M.~Fern{\o}, R.~Juanes, J.W. Both, H.~Class, M.~Delshad, F.~Doster, J.~Ennis-King, J.~Franc, et~al.
\newblock The fluidflower validation benchmark study for the storage of co 2.
\newblock {\em Transp. Porous Media.}, 151(5):865--912, 2024.

\bibitem{PAN201919}
S.-Y. Pan, M.~Gao, K.J. Shah, J.~Zheng, S.-L. Pei, and P.-C. Chiang.
\newblock Establishment of enhanced geothermal energy utilization plans: Barriers and strategies.
\newblock {\em J. Renew. Energy}, 132:19--32, 2019.

\bibitem{WEI2019120}
X.~Wei, Z.-J. Feng, and Y.-S. Zhao.
\newblock Numerical simulation of thermo-hydro-mechanical coupling effect in mining fault-mode hot dry rock geothermal energy.
\newblock {\em J. Renew. Energy}, 139:120--135, 2019.

\bibitem{ASAI2019763}
P.~Asai, P.~Panja, McLennan J., and J.~Moore.
\newblock Efficient workflow for simulation of multifractured enhanced geothermal systems (egs).
\newblock {\em J. Renew. Energy}, 131:763--777, 2019.

\bibitem{zhou2019seismological}
P.~Zhou, H.~Yang, B.~Wang, and J.~Zhuang.
\newblock Seismological investigations of induced earthquakes near the hutubi underground gas storage facility.
\newblock {\em J. Geophys. Res. Solid Earth}, 124(8):8753--8770, 2019.

\bibitem{karev2019geomechanical}
V.I. Karev.
\newblock Geomechanical approach to improving the efficiency of the operation of underground gas storages.
\newblock In V.I. Karev, Dmitry Klimov, and Konstantin Pokazeev, editors, {\em Physical and Mathematical Modeling of Earth and Environment Processes (2018)}, pages 150--158, Cham, 2019. Springer International Publishing.

\bibitem{FIRME2019103006}
P.A.L.P. Firme, D.~Roehl, and C.~Romanel.
\newblock Salt caverns history and geomechanics towards future natural gas strategic storage in brazil.
\newblock {\em J. Nat. Gas Eng.}, 72:103006, 2019.

\bibitem{ellsworth2013injection}
W.L. Ellsworth.
\newblock Injection-induced earthquakes.
\newblock {\em Science}, 341(6142):1225942, 2013.

\bibitem{kikuchi1988contact}
N.~Kikuchi, N.~Kikuchi, and J.T. Oden.
\newblock {\em Contact Problems in Elasticity: A Study of Variational Inequalities and Finite Element Methods}.
\newblock SIAM studies in applied mathematics. SIAM, 1988.

\bibitem{matrins1987}
J.A.C. Martins and J.T. Oden.
\newblock Existence and uniqueness results for dynamic contact problems with nonlinear normal and friction interface laws.
\newblock {\em Nonlinear Anal. Theory Methods Appl.}, 11(3):407--428, 1987.

\bibitem{sofonea2012mathematical}
M.~Sofonea and A.~Matei.
\newblock {\em Mathematical Models in Contact Mechanics}.
\newblock London Mathematical Society Lecture Note Series. Cambridge University Press, 2012.

\bibitem{coussy2004poromechanics}
O.~Coussy.
\newblock {\em Poromechanics}.
\newblock J. Wiley \& Sons, 2004.

\bibitem{berre2019flow}
I.~Berre, F.~Doster, and E.~Keilegavlen.
\newblock Flow in fractured porous media: A review of conceptual models and discretization approaches.
\newblock {\em Transp. Porous Media.}, 130(1):215--236, 2019.

\bibitem{porepy2024}
I.~Stefansson, J.~Varela, E.~Keilegavlen, and I.~Berre.
\newblock Flexible and rigorous numerical modelling of multiphysics processes in fractured porous media using {PorePy}.
\newblock {\em Results Appl. Math.}, 21:100428, February 2024.

\bibitem{berge_finite_2020}
R.~L. Berge, I.~Berre, E.~Keilegavlen, J.M. Nordbotten, and B.~Wohlmuth.
\newblock Finite volume discretization for poroelastic media with fractures modeled by contact mechanics.
\newblock {\em Int. J. Numer. Methods Eng.}, 121(4):644--663, February 2020.

\bibitem{HUEBER20053147}
S.~Hüeber and B.I. Wohlmuth.
\newblock A primal–dual active set strategy for non-linear multibody contact problems.
\newblock {\em Comput. Methods Appl. Mech. Eng.}, 194(27):3147--3166, 2005.

\bibitem{wriggers2005}
P.~Wriggers.
\newblock {\em Computational Contact Mechanics}.
\newblock Springer, 01 2006.

\bibitem{bacq_allatonce_2023}
P.-L. Bacq, S.~Gounand, A.~Napov, and Y.~Notay.
\newblock An all‐at‐once algebraic multigrid method for finite element discretizations of {Stokes} problem.
\newblock {\em Int. J. Numer. Methods Fluids}, 95(2):193--214, February 2023.

\bibitem{notay_new_2016}
Y.~Notay.
\newblock A new algebraic multigrid approach for {Stokes} problems.
\newblock {\em Numer. Math.}, 132(1):51--84, January 2016.

\bibitem{benzi_numerical_2005}
M.~Benzi, G.H. Golub, and J~Liesen.
\newblock Numerical solution of saddle point problems.
\newblock {\em Acta Numer.}, 14:1--137, May 2005.

\bibitem{voet_internodes_2022}
Y.~Voet, G.~Anciaux, S.~Deparis, and P.~Gervasio.
\newblock The {INTERNODES} method for applications in contact mechanics and dedicated preconditioning techniques.
\newblock {\em Comput. Math. Appl.}, 127:48--64, December 2022.

\bibitem{wiesner_algebraic_2021}
T.A. Wiesner, M.~Mayr, A.~Popp, M.W. Gee, and W.A. Wall.
\newblock Algebraic multigrid methods for saddle point systems arising from mortar contact formulations.
\newblock {\em Int. J. Numer. Methods Eng.}, 122(15):3749--3779, August 2021.

\bibitem{franceschini_scalable_2022}
A.~Franceschini, L.~Gazzola, and M.~Ferronato.
\newblock A scalable preconditioning framework for stabilized contact mechanics with hydraulically active fractures.
\newblock {\em J. Comput. Phys.}, 463:111276, August 2022.

\bibitem{franceschini_reverse_2022}
A.~Franceschini, M.~Ferronato, M.~Frigo, and C.~Janna.
\newblock A reverse augmented constraint preconditioner for {Lagrange} multiplier methods in contact mechanics.
\newblock {\em Comput. Methods Appl. Mech. Eng.}, 392:114632, March 2022.

\bibitem{franceschini_block_2019}
A.~Franceschini, N.~Castelletto, and M.~Ferronato.
\newblock Block preconditioning for fault/fracture mechanics saddle-point problems.
\newblock {\em Comput. Methods Appl. Mech. Eng.}, 344:376--401, February 2019.

\bibitem{white_block-partitioned_2016}
J.A. White, N.~Castelletto, and H.A. Tchelepi.
\newblock Block-partitioned solvers for coupled poromechanics: {A} unified framework.
\newblock {\em Comput. Methods Appl. Mech. Eng.}, 303:55--74, May 2016.

\bibitem{mikelic_convergence_2013}
A.~Mikelić and M.F. Wheeler.
\newblock Convergence of iterative coupling for coupled flow and geomechanics.
\newblock {\em Comput. Geosci.}, 17(3):455--461, June 2013.

\bibitem{kim_stability_2011}
J.~Kim, H.A. Tchelepi, and R.~Juanes.
\newblock Stability and convergence of sequential methods for coupled flow and geomechanics: {Fixed}-stress and fixed-strain splits.
\newblock {\em Comput. Methods Appl. Mech. Eng.}, 200(13-16):1591--1606, March 2011.

\bibitem{storvik_optimization_2018}
E.~Storvik, J.W. Both, K.~Kumar, J.M. Nordbotten, and F.A. Radu.
\newblock On the optimization of the fixed-stress splitting for {Biot}'s equations.
\newblock {\em Int. J. Numer. Methods Eng.}, 120(2):179--194, November 2018.

\bibitem{BOTH2017101}
J.W. Both, M.~Borregales, J.M. Nordbotten, K.~Kumar, and F.A. Radu.
\newblock Robust fixed stress splitting for biot’s equations in heterogeneous media.
\newblock {\em Appl. Math. Lett.}, 68:101--108, 2017.

\bibitem{hu_effective_2023}
X.~Hu, E.~Keilegavlen, and J.M. Nordbotten.
\newblock Effective {Preconditioners} for {Mixed}‐{Dimensional} {Scalar} {Elliptic} {Problems}.
\newblock {\em Water Resour. Res.}, 59(1):e2022WR032985, January 2023.

\bibitem{girault_convergence_2016}
V.~Girault, K.~Kumar, and M.F. Wheeler.
\newblock Convergence of iterative coupling of geomechanics with flow in a fractured poroelastic medium.
\newblock {\em Comput. Geosci.}, 20(5):997--1011, October 2016.

\bibitem{porepy2021}
E.~Keilegavlen, R.~Berge, A.~Fumagalli, M.~Starnoni, I~Stefansson, J.~Varela, and I.~Berre.
\newblock Porepy: An open-source software for simulation of multiphysics processes in fractured porous media.
\newblock {\em Comput. Geosci.}, 25:243--265, 2021.

\bibitem{stefansson_fully_2021}
I.~Stefansson, I.~Berre, and E.~Keilegavlen.
\newblock A fully coupled numerical model of thermo-hydro-mechanical processes and fracture contact mechanics in porous media.
\newblock {\em Comput. Methods Appl. Mech. Eng.}, 386:114122, December 2021.

\bibitem{zimmerman1996hydraulic}
R.W. Zimmerman and G.S. Bodvarsson.
\newblock Hydraulic conductivity of rock fractures.
\newblock {\em Transport Porous Med.}, 23:1--30, 1996.

\bibitem{Martin2005}
V.~Martin, J.~Jaffr\'{e}, and J.E. Roberts.
\newblock Modeling fractures and barriers as interfaces for flow in porous media.
\newblock {\em SIAM J. Sci. Comput.}, 26(5):1667--1691, 2005.

\bibitem{barton_strength_1985}
N.~Barton, S.~Bandis, and K.~Bakhtar.
\newblock Strength, deformation and conductivity coupling of rock joints.
\newblock {\em Int. J. Rock Mech. Min. Sci.}, 22(3):121--140, June 1985.

\bibitem{Aavatsmark2002}
I.~Aavatsmark.
\newblock An introduction to multipoint flux approximations for quadrilateral grids.
\newblock {\em Comput. Geosci.}, 6:405--432, 2002.

\bibitem{Nordbotten2016}
J.M. Nordbotten.
\newblock Stable cell-centered finite volume discretization for biot equations.
\newblock {\em SIAM J. Numer. Anal.}, 54(2):942--968, 2016.

\bibitem{ito2003semi}
K.~Ito and K.~Kunisch.
\newblock Semi--smooth newton methods for variational inequalities of the first kind.
\newblock {\em ESAIM: Math. Model. Numer. Anal.}, 37(1):41--62, 2003.

\bibitem{zabegaev_2025_14609885}
Y.~Zabegaev, E.~Keilegavlen, I.~Berre, and K.~Kumar.
\newblock Runscripts for the manuscript ``an effificent preconditioner for mixed-dimensional contact poromechanics based on the fixed stress splitting scheme'', January 2025.

\bibitem{gmsh}
C.~Geuzaine and J.-F. Remacle.
\newblock Gmsh: A 3-d finite element mesh generator with built-in pre- and post-processing facilities.
\newblock {\em Int. J. Numer. Methods Eng.}, 79(11):1309--1331, 2009.

\bibitem{petsc4py}
L.D. Dalcin, R.R. Paz, P.A. Kler, and A.~Cosimo.
\newblock Parallel distributed computing using python.
\newblock {\em Adv. Water Resour.}, 34(9):1124--1139, 2011.
\newblock New Computational Methods and Software Tools.

\bibitem{petsc_user_manual}
S.~Balay, S.~Abhyankar, M.F. Adams, S.~Benson, J.~Brown, et~al.
\newblock {PETSc/TAO} users manual.
\newblock Technical Report ANL-21/39 - Revision 3.19, Argonne National Laboratory, 2023.

\bibitem{hypre}
R.D. Falgout and U.M. Yang.
\newblock hypre: A library of high performance preconditioners.
\newblock In {\em International Conference on computational science}, pages 632--641. Springer, 2002.

\end{thebibliography}
\bibliographystyle{unsrt}

\end{document}